\documentclass[11pt]{article}
\usepackage[english]{babel}
\usepackage[letterpaper,top=2cm,bottom=2cm,left=3cm,right=3cm,marginparwidth=1.75cm]{geometry}

\usepackage{amsmath,amssymb,mathrsfs,amssymb,amsthm}
\usepackage{enumerate}
\usepackage{cases}
\usepackage{graphicx}
\usepackage[colorlinks=true, allcolors=blue]{hyperref}
\usepackage{color}
\newcommand{\vep}{\varepsilon}
\newcommand{\R}{\mathbb R}
\newcommand{\CC}{\mathbb C}
\definecolor{HW}{rgb}{0,0,0}
\definecolor{HW1}{rgb}{0,0,0}
\definecolor{HW2}{rgb}{0,0,0}
\definecolor{HW3}{rgb}{0,0,0}
\definecolor{HW4}{rgb}{0,0,0}

\numberwithin{equation}{section}
\numberwithin{figure}{section}
\numberwithin{table}{section}
\title{Uniform Resolvent Estimates for Subwavelength Resonators:\\
The Minnaert Bubble Case}
\author{Long Li \thanks {RICAM, Austrian Academy of Sciences, A-4040, Linz, Austria (long.li@ricam.oeaw.ac.at)} \; and Mourad Sini \thanks{RICAM, Austrian Academy of Sciences, A-4040, Linz, Austria (mourad.sini@oeaw.ac.at)}}

\newtheorem{theorem}{Theorem}[section]
\newtheorem{corollary}[theorem]{Corollary}
\newtheorem{lemma}{Lemma}[section]
\newtheorem{remark}{Remark}
\newtheorem{definition}{Definition}

\begin{document}
\date{}
\maketitle
\begin{abstract}
Subwavelength resonators are small scaled objects that exhibit contrasting medium properties (either in intensity or sign) while compared to the ones of a uniform background. Such contrasts allow them to resonate at specific frequencies. There are two ways to mathematically define these resonances. First, as the frequencies for which the related system of integral equations is not injective. Second, as the frequencies for which the related resolvent operator of the natural Hamiltonian, given by the wave-operator, has a pole. 
\bigskip

In this work, we consider, as the subwavelength resonator, the Minneart bubble. 
We show that these two mentioned definitions are equivalent. Most importantly, 
\begin{enumerate}
\item we derive the related resolvent estimates which are uniform in terms of the size/contrast of the resonators. As a by product, we show that the resolvent operators have no scattering resonances in the upper half complex plane while they exhibit two scattering resonances in the lower half plane which converge to the real axis, as the size of the bubble tends to zero. As these resonances are poles of the natural Hamiltonian, given by the wave-operator, and have the Minnaert frequency as their dominating real part, this justifies calling them Minnaert resonances.

\item we derive the asymptotic estimates of the generated scattered fields which are uniform in terms of the incident frequency and which are valid everywhere in space (i.e. inside or outside the bubble). 

\end{enumerate}
The dominating parts, for both the resolvent operator and the scattered fields, are given by the ones of the point-scatterer supported at the location of the bubble. In particular, these dominant parts are non trivial (not the same as those of the background medium) if and only if the used incident frequency identifies with the Minnaert one.

\vspace{.2in}
\end{abstract}

\tableofcontents

\section{Introduction and statement of the main results}

\subsection{The Mathematical model}

 Acoustic wave propagation in bubbly media involves complex interactions governed by the resonant behavior of gas bubbles in a liquid. Theses resonant phenomena underpin a wide range of applications in acoustic metamaterials, underwater acoustics, medical ultrasonic imaging, and oceanography. In this work, we focus on the linearized model, refereed to as a Minnaert bubble model, which captures the essential features of wave propagation in such media; further details can be found in \cite{C-M-P-T-1, C-M-P-T-2}. To proceed, the notation and preliminaries needed for the mathematical formulation of the Minnaert bubble model are introduced below.

Let $y_0$ be any fixed point in $\R^3$. For any $\vep>0$, define $\Omega_\vep:= \{x: x=y_0+\vep(y-y_0), y\in \Omega\}$ and $\Gamma_\vep:=\partial \Omega_\vep$. Here, $\Omega\subset \mathbb R^3$ is an open bounded and connected domain with a $C^2$-smooth boundary $\Gamma:=\partial \Omega$. Let $\Omega_\vep \subset \R^3$ denote a micro-bubble embedded in the homogeneous background medium (see Figure \ref{fi1} for the geometric setting of $\Omega$ and $\Omega_\vep$). The acoustic properties of the medium generated by $\Omega_\vep$ and the homogeneous background are characterized by the mass density $\rho_\vep$ and the bulk modulus $k_\vep$, where $\rho_\vep$ and $k_\vep$ are defined by 
\begin{align} \label{eq:41}
\rho_\vep(x) := 
\begin{cases}
\rho_0,               &  x \in \R^3 \backslash \Omega_\vep, \\
{\rho_1}{\vep^{2}},  & x \in \Omega_\vep,
\end{cases}
\quad 
k_\vep(x):=
\begin{cases}
k_0,            &  x \in \R^3 \backslash \Omega_\vep, \\
k_1 \vep^{2},  & x \in \Omega_\vep.
\end{cases}
 \end{align}
Here, $\rho_0, k_0, \rho_1$ and $k_1$ are all positive real numbers. 
\begin{figure}[htbp]
\centering
\includegraphics{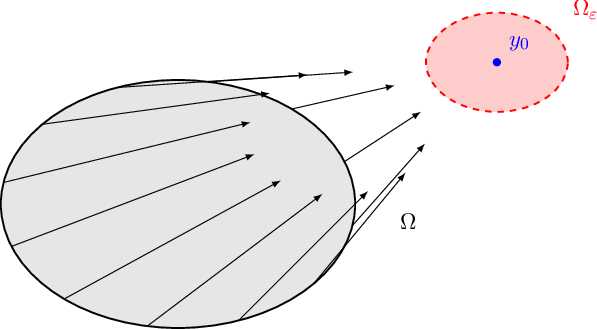}
\caption{Geometric setting of $\Omega$ and $\Omega_\vep$}\label{fi1}
\end{figure}
We use a time-harmonic non-vanishing acoustic wave $u_\omega^{in} $ as an incoming incident wave onto $\Omega_\vep$, i.e., a solution of 
\begin{align*}
\nabla \cdot \frac{1}{\rho_0} \nabla u_\omega^{in}  + \omega^2 \frac{1}{k_0} u^{in}_\omega = 0 \quad \; \textrm{in}\; \R^3,
\end{align*}
where $\omega >0$ is a given incident frequency. For instance, $u_\omega^{in}$ could be a plane wave or a Herglotz wave, which is a superposition of plane waves. Then the scattering of the time-harmonic acoustic waves by the micro-bubble can be mathematically formulated as the problem of finding the total field $u_{\omega,\vep} $ such that 
\begin{align}
&\nabla \cdot \frac 1 {\rho_\vep} \nabla u_{\omega,\vep} +\omega^2 \frac{1}{k_\vep}u_{\omega,\vep} = 0\; \quad\quad\quad\;\;\;\;\; \;\;\;\;\;\; \;\; \; \;\; \; \text{in}\; \R^3,\label{eq:1}\\
& u_{\omega,\vep} = u^{sc}_{\omega,\vep} + u^{in}_\omega \;\;\;\;\;\;\; \quad \qquad \qquad \qquad \qquad \qquad \text{in}\; \mathbb R^3,\\
&\lim_{|x|\rightarrow +\infty}\left(\frac{x}{|x|}\cdot\nabla-i\frac{\omega}{c_0}\right)u^{sc}_{\omega,\vep} = 0. \label{eq:2}
\end{align}
Here, $c_0:=\sqrt{k_0/\rho_0}$ denotes the speed of sound in the background medium, and $\nu$ denotes the outward normal to $\Gamma_\varepsilon$. The equation \eqref{eq:1} is understood as 
\begin{align*}
& \Delta u^+_{\omega,\vep} + \frac{\omega^2\rho^2_0}{k^2_0} u^+_{\omega,\vep} = 0  \qquad\qquad\qquad\;\;\;\;\; \;\;\;\;\;\; \;\; \; \; \;\; \; \text{in}\;\; \R^3 \backslash\overline{\Omega_\vep},\\
& \Delta u^-_{\omega,\vep} + \frac{\omega^2\rho^2_1}{k^2_1} u^-_{\omega,\vep} = 0 \qquad\qquad \;\quad\;\; \;\;\;\;\;\; \;\;\;\;\;\; \;\; \; \; \;\; \; \text{in}\;\; \Omega_\vep, \\
& u^+_{\omega,\vep} = u^-_{\omega,\vep}, \quad  \frac{1}{\rho_0}\partial_\nu u^+_{\omega,\vep} = \frac{1}{\rho_1\vep^2}\partial_\nu u^-_{\omega,\vep}\qquad \;\;\;\;\;\;\;\;\;\text{on}\; \Gamma_\vep. 
\end{align*}
The unique solvability of the above scattering problem \eqref{eq:1}--\eqref{eq:2} for fixed $\vep$ is well known (see, e.g., \cite{DK19, N01}).

The distribution of the {scattering resonances} of general scatterers, i.e. the eventual poles of the related resolvent operators, has been extensively studied, see for instance \cite{J19, LSW, P-V:1919, T-Z-1998}, with the references therein, and the book \cite{DM} for the theoretical studies. The case of a fixed-size bubble with moderate contrast constitutes a specific example within this broader framework. For studies focusing on microdisks, we refer to \cite{BDM-21,Cap12,HPV07}. The computational aspects of scattering resonances are also considered and studied, see \cite{G-H-R:23, LSW, M-N-N:17, M-S:19} and the cited literature therein.

In the present work, we deal with subwavelength resonators, i.e. small but highly contrasting heterogeneities, in the regime (\ref{eq:41}) where the parameter $\vep$ is small. In practice, the mass density and the bulk modulus are very small quantities. Therefore, the bubbles are designed, with a chosen radius $\epsilon$, so that the mass density and the bulk modulus scale as $\epsilon^2$. The $\vep^2$ scaling of the mass density and the bulk modulus in \eqref{eq:41} ensures that the resulting subwavelength resonance frequency is of order $O(1)$ (see formula \eqref{eq:45}). Such a scaling enables the manipulation of resonance phenomena at accessible frequencies, facilitating applications in wave control and materials engineering, see for instance \cite{AFGLZ-17, AZ-17, MS-241,MS-242}.
We believe that our argument can be similarly applied with less effort to other subwavelength resonators that have moderate mass density and large bulk modulus, where a sequence of resonances will be excited (see, e.g., \cite{DGS-21, MMS-18}). This is because the analysis of Minnaert bubbles is more involved, as we have to handle both operators appearing in the used Lippmann-Schwinger equations.

It should be remarked that wave propagation in high-contrast media with small inclusions is also intimately linked to cloaking via transformation optics, where the effects of the small inhomogeneities  are used to assess the effectiveness of approximate cloaking through change-of-variables techniques. For related results and developments in this direction, we refer the reader to \cite{Cap12, KSW08, Ngu12, NT19, NV09, NV12} and the references therein.

We conclude this section by briefly outlining the structure of the remainder of the introduction. Sections \ref{section:1.2} and \ref{section:1.3} summarize the uniform asymptotic results for the acoustic fields (Theorem \ref{th:1}) and the resolvents (Theorems \ref{th:2} and \ref{th:6}), respectively, demonstrating the contribution of Minnaert resonances. Section \ref{section:1.4} compares our results with related works and highlights the main contributions.

\subsection{The Minnaert frequency and the acoustic fields} \label{section:1.2}
Based on the Lippmann-Schwinger equation (see \cite{DGS-21}), the total field $u_{\omega,\vep}$ has the following integral representation 
\begin{align}\label{eq:17}
& u_{\omega,\vep}(x)= u_\omega^{in}(x) +\left(\frac{1}{c^2_1} - \frac{1}{c^2_0} \right)\omega^2\int_{\Omega_\vep} \frac{e^{i{\omega}|x-y|/c_0}}{4\pi|x-y|}u_{\omega,\vep}(y)dy\notag\\
&\quad\quad \quad \;-\left(\frac{\rho_0}{\rho_1\vep^2}-1\right)\int_{\Gamma_\vep}\frac{e^{i{\omega}{|x-y|}/c_0}}{4\pi|x-y|} \partial_\nu u_{\omega,\vep}(y)d\sigma(y) , \quad x\in \R^3 \backslash \Gamma_\vep,
 \end{align}
where $\partial_\nu u_{\omega,\vep}(x):=\lim_{\eta\rightarrow +0} \nu(x) \cdot \nabla u_{\omega,\vep}(x-\eta\nu(x)),\; x\in\Gamma_\vep$ and $c_1: = \sqrt{k_1/\rho_1}$ denotes the speed of sound in the bubble. Based on the above integral expression, it is evident that the total field $u_{\omega,\vep}$ in $\R^3\backslash \Gamma_\vep$ can be fully computed using the value $u_{\omega,\vep}$ within $\Omega_\vep$ and the normal derivative $\partial_\nu u_{\omega, \vep}$ on $\Gamma_\vep$. These two quantities are determined by the succeeding system of integral equations
\begin{align}\label{eq:44}
&u_{\omega,\vep}(x)= u_\omega^{in}(x) +\left(\frac{1}{c^2_1}-\frac{1}{c_0^2}\right)\omega^2\left(N_{\Omega_{\vep},
{\omega}/{c_0}} u_{\omega,\vep}\right)(x) \notag\\ 
&\qquad \;\;\;\;\qquad \qquad -\left(\frac{\rho_0}{\rho_1\vep^2}-1\right)\int_{\Gamma_\vep}\frac{e^{i{\omega}{|x-y|}/c_0}}{4\pi|x-y|} \partial_\nu u_{\omega,\vep}(y)d\sigma(y), \quad
 x\in \Omega_\vep
\end{align}
and 
\begin{align} \label{eq:49}
&\frac{\rho_0}{\rho_1\vep^2}\left(\frac 12\left(1 + \frac {\rho_1\vep^2} {\rho_0}\right) + \left(1-\frac {\rho_1\vep^2} {\rho_0}\right)K_{\Gamma_\vep, \omega/c_0}^*\right) \partial_\nu u_{\omega,\vep}  \notag\\
& \qquad \qquad \qquad \qquad \qquad \qquad \qquad = \partial_\nu u_\omega^{in} + \left(\frac{1}{c^2_1}- \frac{1}{c^2_0}\right)\omega^2\partial_\nu N_{\Omega_{\vep},\omega/c_0} u_{\omega,\vep} \quad  \textrm{on}\; \Gamma_\vep.
\end{align}
Here, the Newtonian operator $N_{\Omega_{\vep},\omega}$ is defined by
\begin{align*}
N_{\Omega_\vep,\omega}: L^2(\Omega_\vep) \rightarrow H_{\textrm{loc}}^2(\R^3), \quad \left(N_{\Omega_\vep,\omega}\phi\right)(x) &:= \int_{\Omega_\vep}\frac{e^{i\omega|x-y|}}{4\pi|x-y|}\phi(y)dy, \quad x\in \R^3,
\end{align*}
and the surface-type operator $K_{\Gamma_\vep,\omega}^*$ is defined by 
\begin{align*}
K_{\Gamma_\vep,\omega}^*: H^{-\frac12}(\Gamma_\vep) \rightarrow H^{\frac 12}(\Gamma_\vep), \quad \left(K_{\Gamma_\vep,\omega}^{*}\phi\right)(x) &:= \partial_{\nu_x}\int_{\Gamma_\vep} \frac{e^{i\omega|x-y|}}{4\pi|x-y|}\phi(y)d\sigma(y),\quad x\in \Gamma_\vep.
\end{align*}
With $\Gamma_\vep$ replaced by $\Gamma$, we write $K^*_{\Gamma, \omega}:= K^*_{\omega}$.
We note that equation \eqref{eq:49} is derived by applying the outward normal derivative to both sides of equation \eqref{eq:17} at any point $x\in \Gamma_\vep$ and using the jump relations of the double layer potential. 

When the size $\vep$ is much smaller than 1, the bubble exhibits high contrast in both its mass density and bulk modulus compared to the homogeneous background medium. It is well-known that this high contrast allows the bubble to resonate at a certain incident frequency, known as the Minnaert frequency, thereby amplifying the scattered field $u_{\omega,\vep}^{sc}$. This phenomenon can be intuitively observed from the integral equations \eqref{eq:17}--\eqref{eq:49}. As $K_{\Gamma_\vep,\omega}^*$ scales approximately as $K_{\Gamma_\vep,\omega}^* \approx -\mathbb I/2$ when $\vep \rightarrow +0$, selecting an appropriate value of $\omega$ would excite the eigenvalue $-1/2$ of $K^*_{0}$, generating a singularity in \eqref{eq:49}. This leads to a very large solution of the system \eqref{eq:17}--\eqref{eq:49}. Mathematically, \cite{AZ18} rigorously derived for the first time a formula for the Minnaert frequency of arbitrarily shaped bubbles by employing layer potential techniques and Gohberg-Sigal theory. They further obtained the asymptotic approximation of the bubble in the far field zone, demonstrating the enhancement of scattering at the Minnaert frequency. Such enhancement was used in different topics ranging from imaging to materials sciences, see \cite{ACCS-20, AFGLZ-17, AZ-17, DGS-21, GS-21, Mukh-Si:2023, SS-2024}. Recently, the authors of \cite{MPS} derived the asymptotic expansion of the scattered field uniform in space (both at near and far zones) by using the resolvent analysis of related frequency-dependent Hamiltonian of Schr\"{o}dinger type. However, the global-in-space asymptotic expansion in \cite{MPS} necessitates an additional frequency constraint, specifically, the incident frequency needs to be outside a narrow vicinity of the Minnaert frequency.

In the current work, we are interested in the uniform asymptotic expansion of the scattered field, both in space and frequency.
Let
\begin{align} \label{eq:45}
\omega_M:= \sqrt{\frac {\mathcal C_\Omega k_1}{|\Omega|\rho_0}}
\end{align}
denote the related Minnaert frequency generated by the micro-bubble,
where $\mathcal C_\Omega$, defined by 
\begin{align} \label{eq:74}
\mathcal C_\Omega:= \int_\Gamma \left(S^{-1}_01\right)(x) d\sigma(x),
\end{align}
represents the capacitance of $\Omega$. Here, $S_0^{-1}$ denotes the inverse of the single layer boundary operator with a kernel of $1/{4\pi |x-y|}$. 
We shall prove
\begin{theorem}\label{th:1}
Let $I\subset \R_+$ be a bounded interval containing $\omega_M$ given by \eqref{eq:45}. Assume that $\alpha > 1/2$ and $\vep>0$. We have
\begin{align}\label{eq:4}
    u^{sc}_{\omega,\vep}(x) = \frac{\varepsilon \omega^2\mathcal C_\Omega}{\omega^2_M-\omega^2-i\varepsilon\frac{\omega^3 \mathcal C_\Omega}{4\pi c_0}}u_\omega^{in}(y_0)\frac{e^{i{\omega}|x-y_0|/c_0}}{4\pi|x-y_0|} + u^{res}_{\omega,\vep}(x)
\end{align}
with 
\begin{align}\label{eq:8}
\|u^{res}_{\omega,\vep}\|_{L^2_{-\alpha}(\mathbb R^3)} \le C_{d_{I,\max},d_{I,\min}}\frac{\varepsilon^{3/2}}{\left|\omega^2_M-\omega^2-i\varepsilon\frac{\omega^3 \mathcal C_\Omega}{4\pi c_0}\right|}, \quad \vep \rightarrow 0,
\end{align}
holding uniformly with respect to all $\omega \in I$.
Here, $d_{I,\max}:= \max_{z\in I} |z|$, $d_{I,\min}:=\min_{z\in I} |z|$ and $C_{d_{I,\max},d_{I,\min}}$ is a constant independent of $\vep$ and $\omega$. In addition, the weighted space $L_{-\alpha}^2(\R^3)$ is defined by $L_{-\alpha}^2(\R^3) :=\left\{u\in L^2_{\textrm{loc}}\left(\R^3\right): (1+|x|^2)^{-\alpha/2} u(x) \in L^2\left(\R^3\right)\right\}$.
\end{theorem}

The above theorem provides, for the first time, the asymptotic expansion of the scattered field uniform in space and frequency. From this result, it is evident that there is a scattering enhancement near the Minnaert frequency, accompanied by a transition from asymptotically trivial to non-trivial scattering as $\omega$ approaches to the the Minnaert frequency $\omega_M$. Notably, since the scattered field satisfies Sommerfeld radiation condition, our result can also be conveniently expressed in the near and far field zones. A key reason why we could avoid assuming the incident frequency $\omega$ to be away from $\omega_M$, as in \cite{MPS}, is that we utilize a novel operator representation \eqref{eq:72} based on the spectral properties of $K^*_0$ to estimate the inverse of operators instead of using Born series inversion methods (see the paragraph before Lemma \ref{le:2}) for more explanations.

Our uniform-in-space asymptotic expansions are uniformly valid with respect to the frequency only in any compact interval, i.e. our analysis doesn't cover the high-frequency regime. This limitation arises from the estimation of inverse of the operator $1/2 (1+\rho_1\vep^2/\rho_0)\mathbb I + (1/2-{\rho_1}\vep^2/{\rho_0})K_{\vep \omega}^*$, which relies on an expansion of $K_z^*$ with respect to the complex parameter $z$. To ensure that higher-order terms in $\vep$ negligible compared, we require that the frequency $\omega$ stays bounded. Moreover, since we consider the non-attenuation medium (similar to the loseloss layer setting in \cite{Ngu12}), the field may exhibit different behaviors, potentially due to the presence of high-frequency resonances.

\subsection{The Minnaert frequency and the resolvent of the acoustic propagator} \label{section:1.3}

\subsubsection{The associated scaled Hamiltonian}

Given $\vep >0$, consider the following natural Hamiltonian $H_{\rho_\vep,k_\vep}$
\begin{align} \label{eq:42}
H_{\rho_\vep, k_\vep} \psi := {k_\vep}\nabla \cdot \frac{1}{\rho_\vep} \nabla \psi
\end{align}
with the domain 
\begin{align} \label{eq:43}
D(H_{\rho_\vep,k_\vep}):= \left\{u\in H^1(\R^3): k_\vep \nabla \cdot \frac{1}{\rho_\vep} \nabla u \in L^2(\R^3)\right\},
\end{align}
where $\rho_\vep$ and $k_\vep$ are given by \eqref{eq:41}. Here, the derivatives in \eqref{eq:42} and \eqref{eq:43} are to be understood in the distributional case. The Hamiltonian $H_{\rho_\vep, k_\vep}$ is a self adjoint operator on $D(H_{\rho_\vep, k_\vep})$ with respect to the scalar product 
\begin{align*}
\langle \phi, \psi \rangle:= \int_{\R^3} \left(k_\vep(x)\right)^{-1} \phi(x) \overline {\psi(x)} dx, \quad \textrm{for} \; \phi, \psi \in D(H_{\rho_\vep,k_\vep}).
\end{align*}
It is known that given fixed $\vep>0$, the resolvent of $H_{\rho_\vep,k_\vep}$
\begin{align*}
R^H_{\rho_\vep, k_\vep}(z) := (-H_{\rho_\vep, k_\vep}-z^2)^{-1}
\end{align*}
is a linear bounded operator mapping from $L^2(\R^3)$ to $H^{1}(\R^3)$ for $z \in \mathbb C_+:= \{z\in\CC: \textrm{Im}(z)>0\}$. 
For the case when $z\in \R \backslash \{0\}$, the corresponding resolvent is defined by
\begin{align*}
R^{H}_{\rho_\vep,k_\vep}(z):= \lim_{\delta \rightarrow 0} (-H_{\rho_\vep, k_\vep}-(z+i\delta)^2)^{-1}.
\end{align*}
The above limit exists, according to the limiting absorption principle (see \cite{KK12,W91} for instance), which can be understood in the following sense
\begin{align*}
\lim_{\delta \rightarrow 0} (-H_{\rho_\vep, k_\vep}-(z+i\delta)^2)^{-1} : L^2_{\alpha}(\R^3) \rightarrow L^2_{-\alpha}(\R^3), \quad \textrm{for}\; z \in \R \backslash \{0\}, \quad  \alpha > \frac{1}2,
\end{align*} 
where the weighted space $L_{\alpha}^2(\R^3)$ is defined by
\begin{align*}
L_\alpha^2(\R^3) := \left\{u\in L^2_{\textrm{loc}}\left(\R^3\right): (1+|x|^2)^{\frac\alpha2} u(x) \in L^2\left(\R^3\right)\right\} \quad \textrm{for}\; \alpha \in \R.
\end{align*}
It is essential to highlight that the Hamiltonian $H_{\rho_\vep, k_\vep}$ and the scattering problem \eqref{eq:1}--\eqref{eq:2} are intimately related. Indeed, for each fixed $\vep >0$ and $\omega > 0$, the kernel of the corresponding resolvent $R^H_{k_\vep, \rho_\vep}$ is nothing but the Green's function corresponding to the scattering problem \eqref{eq:1}--\eqref{eq:2}.

On the other hand, it is worth mentioning that the Hamiltonian $H_{\rho_\vep, k_\vep}: \mathcal H \rightarrow \mathcal H$ with a domain $\mathcal D \subset \mathcal H$ is a black box Hamiltonian for each fixed $\vep>0$ (see Lemma 2.3 and Remark 2.4 in \cite{LSW} for more details). Here,  $\mathcal H$ and $\mathcal D$ are defined by
\begin{align*}
&\mathcal H := \left\{u \in L^2(\R^3): \int_{\R^3} (k_\vep (x))^{-1} |u(x)|^2 dx < +\infty\right\} \quad \mathrm{and}\\
& \mathcal D:= \bigg\{u \in L^2(\R^3): u \in H^1(\R^3\backslash \overline{\Omega_\vep}),\; \nabla\cdot\rho_0^{-1}\nabla u \in L^2(\R^3\backslash \overline{\Omega_\vep}),\\
& \qquad \qquad \qquad \qquad \quad u \in H^1({\Omega_\vep}),\; \nabla\cdot\rho_1^{-1}\vep^{-2}\nabla u \in L^2({\Omega_\vep}),\\
&\qquad \qquad \qquad \qquad \quad u_+ = u_-,\; {\rho_0}^{-1}\partial^+_\nu u = {\rho^{-1}_1\vep^{-2}}\partial^-_\nu u\bigg\}, 
\end{align*}
respectively. We note that $\mathcal D = D\left({H_{\rho_\vep, k_\vep}}\right)$, with $D\left({H_{\rho_\vep, k_\vep}}\right)$ defined in \eqref{eq:43}. It is well established that $R_{\rho_\vep, k_\vep}^{H}(z)$ is a meromorphic family of operators mapping from $\mathcal H_{\textrm{comp}}$ to $\mathcal D_{\textrm{loc}}$ for $z\in \CC$ (see Theorem 4.4 in \cite{DM}), where 
\begin{align} 
&\mathcal H_{\textrm{comp}}: = \{\phi\in \mathcal H: \phi|_{\R^3\backslash B_{R_0}} \in  L^2_{\textrm{comp}}(\R^3 \backslash B_{R_0}) \}, \label{eq:130}\\
&\mathcal D_{\textrm{loc}}: = \{\phi \in \mathcal H: \phi |_{\R^3\backslash B_{R_0}} \in  L^2_{\textrm{loc}}(\R^3 \backslash B_{R_0})\; \textrm{and}\; \chi \phi \in \mathcal D \; \textrm{if}\; \chi \in C_c^{\infty}(\R^3)\; \textrm{and}\; \chi|_{B_{R_0}} =1 \} . \label{eq:104}
\end{align}
Here, $L^2_{\textrm{comp}}(\R^3):=\{u\in L^2(\R^3): \exists R>0, |u(x)| = 0\; \textrm{for}\; |x|>R\}$, $B_{R_0}:=\{x\in \R^3: |x|< R_0\}$ with $R_0$ chosen to be large enough such that $\overline{\Omega_\vep} \subset B_{R_0}$.
This leads to the following definition.
\begin{definition} \label{d2}
We call $z$ a scattering resonance of the Hamiltonian $H_{\rho_\vep, k_\vep}$ if it is a pole of the meromorphic extension of $R^H_{\rho_\vep,k_\vep}(z)$.
\end{definition} 
For more details on the black box Hamiltonian, we refer to \cite[section 4]{DM}. With the $e^{-izt}$ convention, scattering resonances - the poles of the meromorphically continued resolvent in the lower-half complex $z$- plane - govern the decaying oscillatory components of the evolution generated by the associated Hamiltonian: the real part of a scattering resonance determines the oscillation
frequency, while the imaginary part is related to the decay order/ lifetime of the damped oscillation, see \cite{DM} for a comprehensive account. In our setting, all scattering resonances of the Hamiltonian $H_{\rho_\vep,k_\vep}$ lie in the open lower half-plane. It should be remarked that alternative characterizations of resonance based on the field behavior have been widely used in cloaking scenarios \cite{Ngu12, NT19}, and in negative-index media \cite{NN15}. For a real-axis resonance, its life time is formally infinite, hence lifetime is not an informative descriptor, and a field-behavior characterization is more appropriate as adopted in \cite{NN15}.

In the present work, we provide an alternative definition of the scattering resonance (see Definition \ref{d1} in section \ref{app}), which we {have} shown to be equivalent to Definition \ref{d2}, and further establish the relationship between the Minnaert frequency $\omega_M$ and the scattering resonances. {Specifically, we demonstrate that the resolvent of the Hamiltonian $H_{\rho_\vep,k_\vep}$ exhibits two scattering resonances (known as Minnaert resonances) in the lower half complex plane which converge to $\pm \omega_M$, respectively, as the size of the bubble tends to zero} (see {Statement \eqref{z2} of Lemma \ref{le:a1}} and Remark \ref{re:1}). 

Since the Hamiltonian $H_{\rho_\vep, k_\vep}$ depends on the parameter $\vep$, we are interested in the asymptotic behavior of its resolvent $R^H_{\rho_\vep,k_\vep}(z)$ as $\vep\rightarrow 0$. To do so, we proceed to introduce another Hamiltonian $H_{\rho_0,k_0}:= k_0\nabla \rho^{-1}_0\nabla$ with the domain $D(H_{\rho_0,k_0}) := H^2(\R^3)$. Here, $\rho_0$ and $k_0$ are mass density and bulk modulus in the homogeneous background medium, respectively. It is well known that for $z\in \overline{\CC_+} \backslash \{0\}$, $R^H_{\rho_0,k_0}(z):= \left(-H_{\rho_0,k_0} - z^2\right)^{-1}$ acts as a linear bounded mapping from $L^2_{\alpha}(\R^3)$ to $L^2_{-\alpha}(\R^3)$ with $\alpha > 1/2$ (see, e.g., \cite{RT,KK12}), and satisfies 
\begin{align*}
R^H_{\rho_0, k_0}(z) = -c_0^{-2} R_{z/c_0}.
\end{align*}
Here, the operator $R_z$ has the integral representation
\begin{align} \label{eq:69}
\left(R_z\phi\right)(x):= \int_{\R^3} \frac{e^{iz|x-y|}}{4\pi|x-y|} \phi(y) dy, \quad x\in \R^3,\quad z\in \CC.
\end{align} 
In addition, note that due to the relation of $R^H_{\rho_0, k_0}(z)$ and $R_{z/c_0}$, $R_{\rho_0,k_0}^H(z)$ admits an analytic continuation from $\CC_+$ into $\CC$ as a mapping from $L^2_{\textrm{comp}}(\R^3)$ to $L^2_{\textrm{loc}}(\R^3)$.

In the following theorem, we shall present the uniform valid asymptotics of the resolvent of the operator $R^H_{\rho_\vep,k_\vep}(z)$ with respect to $\vep \in \R_+$ and $z$ in any bounded closed subset of $\overline{\CC_+}\backslash \{0\}$, which are closely related to $R^H_{\rho_0, k_0}(z)$.

\begin{theorem}\label{th:2}
Let $\vep>0$, $\alpha>1/2$ and $\omega_M$ be given by \eqref{eq:45}. Suppose that $V$ is a bounded closed subset of $\overline{\CC_+}\backslash \{0\}$. The following expansions hold true.
\begin{enumerate}[(1)]
\item \label{1} Let $a\in \R_+$. Suppose that $\chi_{a,\vep}(x): = 1$ for $x\in \R^3 \backslash \Omega_\vep$ and $\chi_{a,\vep}(x):= a\vep^2$ for $x\in \Omega_\vep$. For any $h\in L_{\alpha}^2(\R^3)$, we have
\begin{align*}
&\left(R^H_{\rho_\vep,k_\vep}(z)\chi_{a,\vep} h\right)(x) = \left(R^H_{\rho_0,k_0}(z) h\right)(x) \notag\\
& \quad \quad \qquad + \frac{\varepsilon z^2\mathcal C_\Omega}{\omega^2_M-z^2-i\varepsilon\frac{z^3 \mathcal C_\Omega}{4\pi c_0}}\left(R^H_{\rho_0,k_0}(z)h\right)(y_0)\frac{e^{iz|x-y_0|/c_0}}{4\pi|x-y_0|} + \left(R^H_{res}(z)h\right)(x)
\end{align*}
with 
\begin{align*}
\left\| R^H_{res}(z)h\right\|_{L^2_{-\alpha}(\mathbb R^3)} \le C_{d_{V,\max}, d_{V,\min}}\frac{\varepsilon^{3/2}}{\left|\omega^2_M-z^2-i\varepsilon\frac{z^3 \mathcal C_\Omega}{4\pi c_0}\right|}\|h\|_{L^2_{\alpha}(\R^3)},  \quad \vep \rightarrow 0,
\end{align*}
holding uniformly with respect to all $z \in V$. 

\item \label{2}
For any $h\in L_{\alpha}^2(\R^3) \cap H^2_{\textrm{loc}}(\R^3)$, we have
\begin{align*}
&\left(R^H_{\rho_\vep,k_\vep}(z) h\right)(x) = \left(R^H_{\rho_0,k_0}(z) h\right)(x) +\frac{\varepsilon z^2\mathcal C_\Omega}{\omega^2_M-z^2-i\varepsilon\frac{z^3 \mathcal C_\Omega}{4\pi c_0}}\left(R^H_{\rho_0,k_0}(z)h\right)(y_0)\frac{e^{iz|x-y_0|/c_0}}{4\pi|x-y_0|}\notag\\ 
&\qquad \qquad \qquad +\frac{\varepsilon \mathcal C_\Omega}{\omega^2_M-z^2-i\varepsilon\frac{z^3 \mathcal C_\Omega}{4\pi c_0}}h(y_0)\frac{e^{iz|x-y_0|/c_0}}{4\pi|x-y_0|}  + \left(R^H_{res}(z)h\right)(x)
\end{align*}
with 
\begin{align*}
    \left\| R^H_{res}(z)h\right\|_{L^2_{-\alpha}(\mathbb R^3)} \le \frac{C_{d_{V,\max}, d_{V,\min}}\varepsilon^{3/2}}{\left|\omega^2_M-z^2-i\varepsilon\frac{z^3 \mathcal C_\Omega}{4\pi c_0}\right|}\left(\|h\|_{L^2_{\alpha}(\R^3)}+ \|h\|_{H^2(B_1(y_0))}\right), \quad \vep \rightarrow 0,
\end{align*}
holding uniformly with respect to all $z \in V$, where $B_1(y_0):= \{x\in\R^3: |x-y_0|<1\}$.
\end{enumerate}
Here, $d_{V,\max}:= \max_{z\in V} |z|$, $d_{V,\min}:= \min_{z\in V} |z|$ and $C_{d_{V,\max}, d_{V,\min}}$ is a positive constant independent of $\vep$, $z$ and $h$.
\end{theorem}
Define the following operator 
\begin{align*}
&\left(\Delta_{y_0}-z^2\right)^{-1} : L^2_{\alpha}(\R^3) \rightarrow L_{-\alpha}^2(\R^3)\; \textrm{for}\; \alpha > \frac 12,\\
&\left((\Delta_{y_0}-z^2)^{-1} \psi\right) (x) := \frac{-1}{c^2_0}\int_{\R^3} \frac{e^{iz|x-y|/c_0}}{4\pi|x-y|} \psi(y)dy - \frac{i}{c_0z}\frac{e^{iz|x-y_0|/c_0}}{|x-y_0|}\int_{\R^3} \frac{e^{iz|y_0-y|/c_0}}{4\pi|y_0-y|} \psi(y)dy.\end{align*}
This operator belongs to the class of the point perturbations of the free Laplacian. We refer to \cite{AGHH05, MPS} for more details on the point perturbations of the Laplacian. Given $\alpha \in \R$, define the space
\begin{align*}
L^2_{\alpha,y_0}(\R^3) := \{h\in L_{\alpha}^2(\R^3): \exists\; r>0, h(x)=0\; 
\textrm{for}\; |x-y_0|<r\}.
\end{align*}

As a by-product of Theorem \ref{th:2}, the resolvent $R^H_{\rho_\vep, k_\vep}(z)$ has a non-trivial limit if and only if $z$ is equal to the Minnaert frequency $\omega_M$.

\begin{corollary} \label{th:3}
Let $\vep>0$ and $\omega_M$ be given by \eqref{eq:45}. Assume that $z\in \overline{\CC_+} \backslash \{0\}$ and $\alpha>1/2$. For every $h\in L^2_{\alpha,y_0}(\R^3)$, we have
\begin{align*}
\lim_{\vep\rightarrow +0}R^H_{\rho_\vep,k_\vep}(z) h = R^H_{\rho_0,k_0}(z) h \quad \textrm{in}\; L^2_{-\alpha}(\R^3), \quad z \ne \pm \omega_M
\end{align*}
and
\begin{align*}
\lim_{\vep\rightarrow +0} R^H_{\rho_\vep,k_\vep}(z) h = (\Delta_{y_0}-z^2)^{-1}h \quad \textrm{in}\; L^2_{-\alpha}(\R^3), \quad z = \pm \omega_M.
\end{align*}
\end{corollary}

Corollary \ref{th:3} states that the non-trivial limit of the resolvent $R^H_{\rho_\vep, k_\vep}(\pm\omega_M) h$, with $h$ supported away from $y_0$, belongs to a class of point perturbations of the free Laplacian. Interestingly, for the more regular $h$ that is not zero at the point $y_0$, statement \eqref{2} of Theorem \ref{th:2} implies that a different asymptotic behavior of the resolvent $R^H_{\rho_\vep, k_\vep}(\pm\omega_M)h$ occurs as $\vep$ tends to $0$.

\subsubsection{The resolvent of the original  acoustic propagator}
Given $\vep>0$, $z\in \overline{\CC_+}\backslash \{0\}$ and $h\in L_{\alpha}^2(\R^3)$ with $\alpha > 1/2$, consider the resolvent $R_{\rho_\vep, k_\vep}(z) h:=u_{z,\vep}^h$ of the acoustic propagator corresponding to the original scattering problem \eqref{eq:1}--\eqref{eq:2}, where $u_{z,\vep}^h \in \left(L^2_{-\alpha}(\R^3)\cap H^2_{\textrm{loc}}(\R^3 \backslash \Gamma_\vep)\right)$ satisfies
\begin{align*}
&\nabla \cdot \frac{1}{\rho_\vep} \nabla u_{z,\vep}^h +  \frac{1}{k_\vep} z^2 u_{z,\vep}^h = -h \quad \; \textrm{in}\; \R^3.
\end{align*}
Here, the mass density $\rho_\vep$ and the bulk modulus $k_\vep$ are specified in \eqref{eq:41}. The resolvent $R_{\rho_\vep, k_\vep}(z) $ is intimately linked to the resolvent $R_{\rho_\vep, k_\vep}^H(z)$ of the Hamiltonian $H_{\rho_\vep, k_\vep}$. In fact, they are related by 
the equation 
\begin{align}\label{eq:29}
R_{\rho_\vep, k_\vep}(z) h = R^H_{\rho_\vep, k_\vep}(z)\left(k_\vep h\right)\quad \textrm{for each}\; z\in \overline{\CC_+} \backslash \{0\},
\end{align}
implying the equivalence of the mapping properties of $R_{\rho_\vep, k_\vep}(z) $ and $R^H_{\rho_\vep, k_\vep}(z)$ for each fixed $\vep >0$. Consequently, $R_{\rho_\vep,k_\vep}(z)$ can be extended to a meromorphic family of operators mapping from $\mathcal H_{\textrm{comp}}$ to $ \mathcal D_{\textrm{loc}}$ for $z\in \CC$, and it shares the same scattering resonances with $R^H_{\rho_\vep, k_\vep}(z)$. Here, the spaces $\mathcal H_{\textrm{comp}}$ and $ \mathcal D_{\textrm{loc}}$ are defined in \eqref{eq:130} and \eqref{eq:104}, respectively. Moreover, building upon formula \eqref{eq:29}, the uniform asymptotics of the resolvent $R^H_{\rho_\vep, k_\vep}(z)$ directly yield the asymptotic behavior of the resolvent $R_{\rho_\vep, k_\vep}(z)$ as $\vep$ tends to 0, leading to the following two corollaries.

\begin{theorem} \label{th:6}
Let $\vep>0$ and $\omega_M$ be given by \eqref{eq:45}. Suppose that $V$ is a bounded closed subset of $\overline{ \CC_+}\backslash \{0\}$. For any $h\in L_{\alpha}^2(\R^3)$ with $\alpha > 1/2 $, we have
\begin{align*}
&\left(R_{\rho_\vep,k_\vep}(z) h\right)(x) = {k_0}\bigg[\left(R^H_{\rho_0,k_0}(z) h\right)(x) \notag\\
& \quad \quad \qquad + \frac{\varepsilon z^2\mathcal C_\Omega}{\omega^2_M-z^2-i\varepsilon\frac{z^3 \mathcal C_\Omega}{4\pi c_0}}\left(R^H_{\rho_0,k_0}(z)h\right)(y_0)\frac{e^{iz|x-y_0|/c_0}}{4\pi|x-y_0|}\bigg] + \left(R_{res}(z)h\right)(x)
\end{align*}
with 
\begin{align*}
\| R_{res}(z)h\|_{L^2_{-\alpha}(\mathbb R^3)} \le C_{d_{V,\max}, d_{V,\min}}\frac{\varepsilon^{3/2}}{\left|\omega^2_M-z^2-i\varepsilon\frac{z^3 \mathcal C_\Omega}{4\pi c_0}\right|}\|h\|_{L^2_{\alpha}(\R^3)},  \quad \vep \rightarrow 0,
\end{align*}
holding uniformly with respect to all $z \in V$. Here, $d_{V,\max}:= \max_{z\in V} |z|$, $ d_{V,\min}:= \min_{z\in V} |z|$ and $C_{d_{V,\max}, d_{V,\min}}$ is a positive constant independent of $\vep$, $z$ and $h$.
\end{theorem}

\begin{corollary} \label{th:7}
Let $\vep>0$ and $\omega_M$ be given by \eqref{eq:45}. Assume that $z\in \overline{\CC_+} \backslash \{0\}$ and $\alpha > 1/2$. We have
\begin{align*}
\left\|R_{\rho_\vep, k_\vep}(z) - {k_0}R^H_{\rho_0,k_0}(z)\right\|_{L_{\alpha}^2(\R^3), L^2_{-\alpha}(\R^3)} \le C_{|z|} \vep, \quad z \ne \pm \omega_M 
\end{align*}
and
\begin{align*}
&\left\|R_{\rho_\vep, k_\vep}(z) - {k_0} (\Delta_{y_0}-z^2)^{-1}\right\|_{L_{\alpha}^2(\R^3), L^2_{-\alpha}(\R^3)} \le C_{|z|}\vep^{1/2}, \quad z = \pm \omega_M.
\end{align*}
Here, $C_{|z|}$ is a positive constant independent of $\vep$.
\end{corollary}

In comparison with Theorem \ref{th:2} regarding the asymptotic behaviors of $R^H_{\rho_\vep,k_\vep}(z)$, Theorem \ref{th:6} provides a unified asymptotic formula of $R_{\rho_\vep, k_\vep}(z) h$ for all $h \in L_{\alpha}^2(\R^3)$ as $\vep$ tends to $0$. Such asymptotic formula leads to the strong convergence of the resolvent of the original acoustic propagator, as articulated in Corollary \ref{th:7}. This specific difference is directly attributable to equation \eqref{eq:29}.

The uniform resolvent estimates provided in Theorems \ref{th:2} and \ref{th:6} are not universally applicable across all $z \in \CC$ due to the existence of scattering resonances of the Hamiltonian $H_{\rho_\vep,k_\vep}$ in the lower half of the complex plane $\CC_-$. Indeed, in section \ref{app}, we show that the resolvent $R^H_{\rho_\vep,k_\vep}(z)$, with sufficiently small $\vep>0$, exhibits two scattering resonances $z_{\pm}(\vep)$, both situated in the lower half complex plane, converging respectively to $\pm\omega_M$ at the order of $\vep$, as $\vep$ goes to zero (see also Remark \ref{re:1}).

By definition of the resolvent of the original acoustic propagator, Theorem \ref{th:6} also covers a quantitative, uniform-in-space asymptotic expansion of the field radiated by the source in a weighted $L^2$ space. Our analysis employs the boundary- and volume- integral operators. See also \cite{Ngu12, NT19} for PDE-estimate-based treatments of small-inclusion effects under various source excitations; these work allow nonconstant coefficients inside the inclusions and are linked to approximate cloaking.

\subsection{Comparison with related works} \label{section:1.4}
Let us now summarize and highlight the main contributions in comparison to the previous works. 

\begin{enumerate}
\item  First, we derive the asymptotic expansion of the scattered field uniform in space and frequency.  In \cite{MPS}, a global-in-space asymptotic expansion was derived under the condition that the incident frequencies are outside a vicinity of the Minnaert frequency $\omega_M$. Here, we remove this condition.

\item Second, we establish the relationship between the Minnaert frequency $\omega_M$ and the scattering resonance of the natural Hamiltonian $H_{\rho_\vep,k_\vep}$.
It is worth mentioning that the usual characterization of the Minnaert frequency, also known as the Minnaert resonance, was formulated as the frequency where the related system of boundary integral equations fails to be injective, see for instance \cite{AZ18, AZ17,FH}. To the best of our knowledge, it remained unclear which Hamiltonian exhibits the Minnaert resonance as the pole of its resolvent. In this paper, we demonstrate that the Minnaert resonance is actually the scattering resonance of the natural Hamiltonian $H_{\rho_\vep,k_\vep}$ and further construct two sequences of Minnaert resonances in the lower-half complex plane, which converge to $\pm \omega_M$, respectively, as the size of the bubble tends to zero.

\item Third, we derive the related resolvent estimates, which are uniform over the bounded closed subsets of $\overline{\CC_+}\backslash\{0\}$ and with respect to the size/contrast of the resonators.  Resolvent estimates were first derived in \cite{MPS} where a  different Hamiltonian was proposed which is a frequency-dependent Schr\"{o}dinger type operator $H_\omega(\vep)$, that includes a singular $\delta-$ like potential supported at the interface of the bubble (see (1.25)--(1.27) there for the detailed definition of $H_{\omega}(\vep)$). Compared to this, it should be remarked that the Hamiltonian we are considering in the current work is, instead, the natural wave-operator $H_{\rho_\vep, k_\vep}$, see (\ref{eq:42})-(\ref{eq:43}). In \cite{MPS}, the corresponding resolvent $\left(H_{\omega}(\vep) - z^2\right)^{-1}$, for $z\in \mathbb C_+ \backslash i\R_+$ and each fixed $\omega$, was shown to exhibit a non-trivial limit as $\vep$ tends to $0$, if and only if $\omega = \omega_M$, by using singular perturbation methods (see \cite[Theorem 1.1]{MPS}). In contrast, the approach we developed here, which is solely based on the Lippmann-Schwinger equation, is more straightforward. In addition, the natural Hamiltonian $H_{\rho_\vep,k_\vep}$ considered in this paper is intimately linked to the wave propagation in the presence of a subwavelength resonator given by a Minneart bubble in time domain. This connection is further validated in \cite{LS}, which shows that, under excitation by a causal source compactly supported in space and time, the solution exhibits a decaying resonant oscillation whose frequency is the Minnaert resonances' real part and whose lifetime its inversely proportional to the magnitude of its imaginary part; this term dominates on time scales up to that lifetime.
\end{enumerate}

\noindent The remaining part of this work is divided as follows. In section \ref{section-2}, we derive the needed asymptotic estimates of the auxiliary operators that appear in the proofs of the different theorems stated above. In section \ref{section-3} and section \ref{section-4}, we provide the detailed proofs of these theorems. In section \ref{app}, we show how the Minnaert frequency, given in (\ref{eq:45}), is the dominant part of scattering resonances of the Hamiltonian $H_{\rho_\vep,k_\vep}$ in the sense of Definition \ref{d2}. This justifies calling it the Minnaert resonance. 

\section{Asymptotic estimates of auxiliary operators}\label{section-2}
This section is devoted to analyzing asymptotic behaviors of certain operators that play a crucial role in the proofs of Theorems \ref{th:1} and \ref{th:2}. Before proceeding, we introduce some new notations. For two Banach spaces $X$ and $Y$, denote the space of all linear bounded mapping from $X$ to $Y$ by $\mathcal L(X,Y)$. For simplicity, $\mathcal L(X,X)$ is also denoted by $\mathcal L(X)$. Let $D \subset \R^3$ be any open bounded and connected domain with a smooth boundary $\partial D$. For $z\in \CC$, define operators
\begin{align*}
&SL_{\partial D, z}: H^{\frac12}(\partial D) \rightarrow H_{\textrm{loc}}^{2}(\R^3 \backslash \partial D),  \;\; \left(SL_{\partial D, z}\phi\right)(x) := \int_{\partial D} \frac{e^{iz|x-y|}}{4\pi|x-y|}\phi(y)d\sigma(y), \;\; x \in \R^3\backslash \partial D,\\
&S_{\partial D, z}: H^{-\frac 12}(\partial D)\rightarrow H^{\frac12}(\partial D), \quad \left(S_{\partial D, z}\phi\right)(x) := \int_{\partial D} \frac{e^{iz|x-y|}}{4\pi|x-y|}\phi(y)d \sigma(y),\quad x\in \partial D, \\
&N_{ D,z}: L^2( D) \rightarrow H_{\textrm{loc}}^2(\R^3), \quad \left(N_{D, z}\phi\right)(x) := \int_{ D}\frac{e^{iz|x-y|}}{4\pi|x-y|}\phi(y)dy, \quad x\in \R^3,\\
&K_{\partial D,z}^*: H^{-\frac12}(\partial D) \rightarrow H^{\frac 12}(\partial D), \quad \left(K_{\partial D, z}^{*}\phi\right)(x) := \partial_{\nu_x}\int_{\partial D} \frac{e^{iz|x-y|}}{4\pi|x-y|}\phi(y)d\sigma(y),\quad x\in \partial D, 
\end{align*}
\begin{align*}
&K_{\partial D,z}: H^{-\frac12}(\partial D)\rightarrow H^{\frac 12}(\partial D), \quad \left(K_{\partial D,z}\phi\right)(x) := \int_{\partial D} \partial_{\nu_y}\frac{e^{iz|x-y|}}{4\pi|x-y|}\phi(y)d\sigma(y),\quad x\in \partial D.
\end{align*}
$SL_{\partial D,z}$, $S_{\partial D, z}$, $N_{D,z}$ and $K_{\partial D,z}$ are also referred to as the single-layer potential, the single-layer boundary operator, the Newtonian operator and Neumann-Poincar\'{e} operator, respectively. It is known that when $z\in \overline{\CC_+}$, $SL_{\partial D, z} \in \mathcal L \left(H^{1/2}(\partial D), H_{-\alpha}^{2}(\R^3 \backslash \partial D)\right)$ and $N_{D, z} \in \mathcal L \left(L^2(D), H_{-\alpha}^{2}(\R^3)\right)$ for $\alpha>1/2$.  We refer to \cite{WM00} for further details regarding the integral operators mentioned above. For the sake of the notational simplicity, the operators $SL_{\Gamma,z}$, $S_{\Gamma,z}$, $N_{\Omega,z}$, $ K^*_{\Gamma,z}$ and $K_{\Gamma,z}$ will henceforth be denoted by $SL_{z}$, $S_{z}$, $N_{z}$, $ K^*_{z}$ and $K_{z}$, respectively.
Furthermore, we denote by $\gamma$ the operator that maps a function onto its Dirichlet trace. It is well established that the trace operator $\gamma$ satisfies, up to a positive bound $C_\Omega$,
\begin{align}\label{eq:63}
\|\gamma \phi\|_{H^{s-\frac 12}(\Gamma)} \le C_{\Omega}\|\phi\|_{H^s(\Omega)}, \quad s>\frac 12.
\end{align}
 Given that $z$ is not a Dirichlet eigenvalue of $-\Delta$ in $\Omega$, it is established that
\begin{align*}
\left(S_{z}\right)^{-1} \in \mathcal L \left(H^{\frac 12}(\Gamma), H^{-\frac 12}(\Gamma)\right).
\end{align*}
Note that for each $g\in H^2(\Omega)$ solving equation $\Delta g + k^2g = f$ with $f\in L^2(\Omega)$, the normal derivative of $g$ on $\Gamma$ can be represented by 
\begin{align}\label{eq:51}
\partial_\nu g= S_{z}^{-1}\left(\frac 12 \mathbb I+ K_{z}\right)\gamma g + S_{z}^{-1} \gamma N_{z} f \quad \textrm{on}\;\Gamma.
\end{align}
\eqref{eq:51} can be easily derived by using Green formulas and applying the jump relations of the single-layer and double-layer potential (see, e.g., [Theorem 3.1] in \cite{DK19}).
Let $B_R(y_0):=\{x\in \R^3: |x-y_0| < R\}$ denote a ball at $y_0 \in \R^3$ with a radius $R>0$. Define
$H^2_{\alpha}(\R^3):= \{u\in H^2_{\textrm{loc}}\left(\R^3\right): |\nabla^j u| \in L_{\alpha}^2\left(\R^3\right), j\in\{0,1,2\}\}$ for $\alpha \in \R$.
From now on, $\mathbb I $ denotes an identity operator in various spaces, and the constants may be different at different places. 

We first present the expansions of $S_z$, $K_{z}$, $K^{*}_{z}$, $N_z$ and $SL_z$ when $z$ belongs to a bounded subset of $\CC$.

\begin{lemma}\label{le:5}
Let $z$ belong to a bounded subset of $\CC$. The following arguments hold true.
\begin{enumerate}[(a)]
\item \label{b1}  The expansion $S_z$ = $S_0 + \sum_{j=1}^{\infty} z^j S^{(j)}$ is uniformly convergent in $\mathcal L(H^{-1/2}(\Gamma), H^{1/2}(\Gamma))$ with respect to $z$. Here, $S^{(j)}$ is defined by
\begin{align*}
\left(S^{(j)} \phi\right)(x):= \frac{i}{4\pi} \int_{\Gamma}\frac{(i|x-y|)^{(j-1)}}{j!} \phi(y)d\sigma(y), \quad x\in \Gamma.
\end{align*}
\item \label{b2} The expansion $K^*_z$ = $K^*_0 + \sum_{j=1}^{\infty}z^{j} K^{*,(j)}$ is uniformly convergent in $\mathcal L(H^{-1/2}(\Gamma), H^{1/2}(\Gamma))$ with respect to $z$. Here, $K^{*,(j)}$ is defined by
\begin{align*}
\left(K^{*,(j)} \phi\right)(x):= \frac{i^j(j-1)}{4\pi j!} \int_{\Gamma}|x-y|^{j-3}(x-y)\cdot \nu(x) \phi(y) d\sigma(y),  \quad x\in \Gamma.
\end{align*}

\item \label{b3}
The expansion $K_z$ = $K_0 + \sum_{j=1}^{\infty} z^{j}K^{(j)}$ is uniformly convergent in $\mathcal L(H^{-1/2}(\Gamma), H^{1/2}(\Gamma))$ with respect to $z$. Here, $K^{(j)}$ is defined by
\begin{align*}
\left(K^{(j)}\phi\right)(x):= -\frac{i^j(j-1)}{4\pi j!} \int_{\Gamma}|x-y|^{j-3}(x-y)\cdot \nu(y) \phi(y) d\sigma(y), \quad x\in \Gamma.
\end{align*}

\item \label{b4}
The expansion $N_z$ = $N_0 + \sum_{j=1}^{\infty} z^{j}N^{(j)}$ is uniformly convergent in $\mathcal L(L^{2}(\Omega), H^{2}(\Omega))$ with respect to $z$. Here, $N^{(j)}$ is defined by
\begin{align*}
\left(N^{(j)} \phi\right)(x):= \frac{i}{4\pi} \int_{\Omega}\frac{(i|x-y|)^{(j-1)}}{j!} \phi(y)d y, \quad x\in \Omega.
\end{align*}

\item \label{b5}
The expansion $SL_z$ = $SL_0 + \sum_{j=1}^{\infty} z^{j}SL^{(j)}$ is uniformly convergent in $\mathcal L(H^{-1/2}(\Gamma), H^{1}(\Omega))$ with respect to $z$. Here, $SL^{(j)}$ is defined by
\begin{align*}
\left(SL^{(j)} \phi\right)(x):= \frac{i}{4\pi} \int_{\Gamma}\frac{(i|x-y|)^{(j-1)}}{j!} \phi(y)d\sigma(y), \quad x\in \Omega.
\end{align*}
\end{enumerate}
\end{lemma}

\begin{proof}
The asymptotic expansions of operators $S_z$ and $K^*_z$ for the case when $z\in \R$ are detailed in Appendix A in \cite{AZ17}. In a similar way, the asymptotic expansions of $S_z$, $K_z$, $K^*_z$, $N_z$ and $SL_z$ mentioned in this lemma can also be derived.
\end{proof}

As a consequence of Lemma \ref{le:5}, we have the following refinements.

\begin{lemma} \label{le:8}
Let $z \in \CC$. The following arguments hold true.
\begin{enumerate}[(a)]
\item \label{f1}  Assume that $z$ is sufficiently small such that $|z|<1$ and $S_z^{-1}$ exists. We have
\begin{align}
&\bigg|\int_{\Gamma}\big[S_z^{-1}\big(1/2 \mathbb I + K_z\big)1\big](y)d\sigma(y) - z^2 \int_\Gamma (K^{(2)} 1)(y) (S^{-1}_01)(y)d\sigma(y)\bigg| \le C |z|^3, \quad \textrm{as}\; z \rightarrow 0. \label{eq:10}
\end{align}
Furthermore, for $\phi \in \{\psi \in H^{1/2}(\Gamma): \int_{\Gamma}\left(S^{-1}_01\right)(y) \psi (y)d\sigma(y) = 0\}$, we have 
\begin{align} \label{eq:39}
\bigg|\int_{\Gamma}\big[S_z^{-1}\big(1/2 \mathbb I + K_z\big) \phi \big](y)d\sigma(y))\bigg| \le C |z|^2\|\phi\|_{H^{\frac 12}(\Gamma)}, \quad \textrm{as}\; z \rightarrow 0.
\end{align}

\item \label{f2} Assume that $|z|<1$. We have
\begin{align}
&\|N_{z}\|_{L^2{(\Omega)}, H^2(\Omega)} \le C \quad \mathrm{and} \label{eq:116}\\
& \|SL_{z}\|_{H^{-\frac 12}(\Gamma), H^1(\Omega)} \le C.\label{eq:127}
\end{align}
\end{enumerate}
Here, $C$ is a constant independent of $z$. 
\end{lemma}
\begin{proof}
\eqref{f1} It easily follows from statement \eqref{b1} of Lemma \ref{le:5} that 
\begin{align}
\left\|S^{-1}_{z}-S^{-1}_0 - z S^{-1}_0 S^{(1)} S_0^{-1} \right\|_{{H^{\frac 12}(\Gamma),H^{-\frac12}(\Gamma)}
}\le C|z|^2. \label{eq:78}
\end{align}
Employing  statement \eqref{b3} of Lemma \ref{le:5}, we have
\begin{align}\label{eq:79}
\left\| \frac 12 \mathbb I+ K_{z} - \left(\frac 1 2 \mathbb I +K_0 + {z^2}K^{(2)}\right)\right\|_{{H^{\frac12}(\Gamma),H^{\frac12}(\Gamma)}}\le C|z|^3.
\end{align}
We note that for $\phi \in \{\psi \in H^{1/2}(\Gamma): \int_{\Gamma}\left(S^{-1}_01\right)(y) \psi (y)d\sigma(y) = 0\}$, we have  
\begin{align} \label{eq:40}
\int_\Gamma \left(S^{-1}_0\phi \right)(y) d\sigma(y) = 0. 
\end{align}
This, together with the fact that $(1/2 \mathbb I + K_0) 1 = 0$,  inequalities \eqref{eq:78} and \eqref{eq:79} yields \eqref{eq:10}.

Moreover, since $(1/2 \mathbb I + K^*_0) \left(S_0^{-1} 1\right) = 0$,  we have
\begin{align} \label{eq:122}
\int_\Gamma \left(S^{-1}_0 1\right)(y)\left[\left(\frac 12 \mathbb I+ K_0\right) \phi \right](y)d\sigma(y) = 0, \quad\phi \in H^{\frac 12}(\Gamma).
\end{align}
Combining \eqref{eq:78}, \eqref{eq:79}, \eqref{eq:40} and \eqref{eq:122} gives \eqref{eq:39}.

\eqref{f2}. Since $|z|<1$, inequalities \eqref{eq:116} and \eqref{eq:127}  follow from statements \eqref{b4} and \eqref{b5} of Lemma \ref{le:5}, respectively. 
\end{proof}

Let $y_0$ be any fixed point in $\R^3$, and introduce the map
\begin{align}\label{eq:64}
\Phi_\vep(y):= y_0+ \vep(y-y_0), \quad \vep>0.
\end{align}
Given any complex valued function $\phi$ and an operator $\mathcal A$ mapping complex valued functions from one function space to another, we define $\left(\phi\circ \Phi_\vep \right) (y):= \phi(\Phi_\vep(y))$ and $\left(\left(\Phi_\vep \circ \mathcal A\right)\phi\right)(y) = (\mathcal A \phi) (\Phi_\vep(y))$. The following lemma will illustrate the asymptotic behaviors of some functions composed with the map $\Phi_\vep(y)$ as $\vep$ tends to 0. 

\begin{lemma} \label{le:7}
Let $\alpha > 1/2$. Assume that $z \in \overline{\CC_+}\backslash \{0\}$ and $\vep\in \R_+$ such that $\vep< 1/\sup_{x\in \Omega} |x-y_0|$. The following arguments hold true.
\begin{enumerate}[(a)]

\item \label{a1} For $\phi_1\in H_{\textrm{loc}}^2(\R^3)$, we have
\begin{align*}
\|\gamma\left(\phi_1 \circ \Phi_{\vep}-\phi_1(y_0)\right)\|_{H^{\frac 32}(\Gamma)} \le C\vep^{\frac 12}\|\phi_1\|_{H^2(B_1(y_0))}. 
\end{align*}

\item \label{a2} For $\phi_2\in L_{\alpha}^2(\R^3)$, we have
\begin{align} \label{eq:46}
&\|\left(\Phi_{\vep} \circ R_{z}\right) \phi_2 - \left(R_{z} \phi_2\right) (y_0)\|_{H^2(\Omega)} \le C \vep^{\frac 12}
\|R_{z}\|_{L^2_{\alpha}(\R^3),H^2_{-\alpha}(\R^3)
} \|\phi_2\|_{L_{\alpha}^2(\R^3)}.
\end{align}

\item \label{a3} For $\phi_3\in L^2(\Omega)$, we have
\begin{align}\label{eq:47}
&\left(\left(\Phi_{1/\vep} \circ N_{\vep z}\right) \phi_3\right)(y) =  \vep\frac{e^{i z|y-y_0|}}{4\pi|y-y_0|}\int_\Omega\phi_3(x)dx + \textrm{Res}(y)
\end{align}
with $\textrm{Res}(y)$ satisfying
\begin{align}
\|Res\|_{L_{-\alpha}^2(\R^3)}\le C \vep^{\frac 32}\|R_{-\overline z}\|_{L^2_{\alpha}(\R^3),H^2_{-\alpha}(\R^3) 
} \|\phi_3\|_{L^2(\Omega)}. \label{eq:121}
\end{align}

\item \label{a4}
For $\phi_4\in H^{-1/2}(\Gamma)$, we have
\begin{align*}
\left(\left(\Phi_{1/\vep} \circ SL_{\vep z}\right) \phi_4\right)(y) = \vep\frac{e^{i z|y-y_0|}}{4\pi|y-y_0|}\int_\Gamma\phi_4(x)d\sigma(x) + \textrm{Res}(y)
\end{align*}
with $\textrm{Res}(y)$ satisfying
\begin{align*} 
&\|Res\|_{L_{-\alpha}^2(\R^3)} \le C\vep^{\frac 32}\|R_{-\overline z}\|_{L^2_{\alpha}(\R^3),H^2_{-\alpha}(\R^3)
}\|\phi_4\|_{H^{-\frac12}(\Gamma)}.
\end{align*}
\end{enumerate}
Here, $C$ is a constant independent of $\vep$ and $z$.
\end{lemma}

\begin{proof}
\eqref{a1}
It follows from \eqref{eq:63} that
\begin{align*}
\|\gamma\left(\phi_1 \circ \Phi_{\vep}-\phi_1(y_0)\right)\|_{H^{\frac 32}(\Gamma)} \le C\|\phi_1 \circ \Phi_{\vep}-\phi_1(y_0)\|_{H^2(\Omega)}.
\end{align*}
Thus, it suffices to prove 
\begin{align}\label{eq:26}
\|\phi_1 \circ \Phi_{\vep}-\phi_1(y_0)\|_{H^2(\Omega)} \le C\vep^{\frac 12}\|\phi_1\|_{H^2(B_1(y_0))}, \quad \phi_1\in H_{}^2(\R^3).
\end{align}
Denote by $(A_{\vep}\phi_1)(x):=\left(\phi_1 \circ \Phi_{\vep}\right)(x)-\phi_1(y_0) = \phi_1(y_0+\vep(x-y_0))-\phi_1(y_0)$ 
for $x\in \R^3$. Note that 
\begin{align}\label{eq:150}
\Phi_\vep(\Omega) \subset B_1(y_0) \quad \textrm{when} \; \vep < 1/\sup_{x\in \Omega}|x-y_0|,
\end{align}
where $\Phi_{\vep}(\Omega):=\{\Phi_{\vep}(x): x\in \Omega\}$. By using the inequality 
\begin{align} \label{eq:11}
\sup_{x\in \Lambda}|\phi(x)| + \sup_{x,y\in \Lambda,x\ne y}\frac{|\phi(x)-\phi(y)|}{|x-y|^{1/2}} \le C_\Lambda\|\phi\|_{H^{2}(\Lambda)}
\end{align}
for any compact set $\Lambda \subset \R^3$ (see, e.g., [Section 5.6.3] in \cite{E10}), we have 
\begin{align*}
\|A_{\vep}\|^2_{L^2(\Omega)} \le C\vep\|\phi_1\|_{H^2{(B_1(y_0))}}\int_{\Omega}|x-y_0|dx \le C\vep\|\phi_1\|_{{H^2{(B_1(y_0))}}}.
\end{align*}
Furthermore, a straightforward calculation gives that 
\begin{align*}
\|\partial_{x_j}A_{\vep}\|^2_{L^2(\Omega)} & \le \frac {1}{\vep} \int_{\Phi_{\vep}(\Omega)} |\partial_{x_j}\phi_1(x)|^2 dx \le \vep|\Omega|\|\partial_{x_j}\phi_1(x)\|^2_{L^6(\Phi_{\vep}(\Omega))}\\
& \qquad  \qquad \qquad \qquad \qquad \; \qquad\le C\vep \|\phi_1(x)\|^2_{{H^2{(B_1(y_0))}}}, \quad j\in\{1,2,3\}.
\end{align*}
The last inequality follows from \eqref{eq:150} and the fact that $\|\phi\|_{L^6(\Lambda)} \le C_\Lambda \|\phi\|_{H^1(\Lambda)}$ for any compact set $\Lambda \subset \R^3$ (see, e.g., [Section 5.6.3] in \cite{E10})). Moreover, it is easy to verify that 
\begin{align*}
\|\partial_{x_{j_1}x_{j_2}}A_{\vep}\|^2_{L^2(\Omega)} \le {\vep} \int_{\Phi_{\vep}(\Omega)} |\partial_{x_{j_1}x_{j_2}} \phi_1(x)|^2 dx \le {\vep}\|\phi_1(x)\|^2_{H^2{(B_1(y_0))}}, \;\; j_1,j_2 \in \{1,2,3\}.
\end{align*}
Therefore, based on the above discussions, we obtain that \eqref{eq:26} holds. This finishes the proof of this statement.

\eqref{a2}
By \eqref{eq:26}, we have 
\begin{align*}
\|\left(\Phi_{\vep} \circ R_{z}\right) \phi_2 - \left(R_{z} \phi_2\right) (y_0)\|_{H^2(\Omega)} \le C \vep^{\frac 12} \| R_{z} \phi_2\|_{H^2{(B_1(y_0))}},
\end{align*}
whence \eqref{eq:46} follows from the fact $R_z \in \mathcal L\left(L^2_{\alpha}(\R^3),H^2_{-\alpha}(\R^3)\right)$ for the case when $z\in \R\backslash \{0\}$ and $R_z \in \mathcal L\left(L^2(\R^3),H^2(\R^3)\right)$ for the case when $z\in \CC_+$.

\eqref{a3} It is clear that for $y\in \R^3$
\begin{align} \label{eq:53}
\left(\left(\Phi_{1/\vep} \circ N_{\vep z}\right)\phi_3\right)(y) = \left(N_{\vep z} \phi_3\right)(y_0+1/\vep(y-y_0)) = \vep \int_{\Omega}\frac{e^{iz|y-y_0-\vep(x-y_0)|}}{4\pi|y-y_0-\vep(x-y_0)|} \phi_3(x)dx.
\end{align}
By a straightforward calculation, we get
\begin{align}
&\int_{\R^3}\overline{v(y)}\int_{\Omega}\frac{e^{iz|y-y_0 -\vep(x-y_0)|}}{4\pi|y-y_0-\vep(x-y_0)|} \phi_3(x)dx dy = \notag\\
& \qquad \qquad \qquad \qquad \qquad \qquad \int_{\Omega}\phi_3(x) \overline{\int_{\R^3} \frac{e^{-i{\overline z}|y_0+\vep(x-y_0)-y|}}{4\pi|y_0+\vep(x-y_0)-y|} v(y) dy} dx. \label{eq:120}
\end{align}  
Combining \eqref{eq:53} and \eqref{eq:120} gives
\begin{align}\label{eq:59}
\left\langle \left(\Phi_{1/\vep} \circ N_{\vep z}\right) \phi_3, v 
\right\rangle_{L_{-\alpha}^2(\R^3),L_{\alpha}^2(\R^3)} = \langle \phi_3 , \vep\left(\Phi_{\vep} \circ R_{-\overline z}\right)  v\rangle_{L^2(\Omega),L^2(\Omega)}.
\end{align}
Further, we find
\begin{align}
\left\langle \frac{e^{i z|\cdot-y_0|}}{4\pi|\cdot-y_0|}\int_\Omega\phi_3(x)dx,v(\cdot) \right\rangle_{L_{-\alpha}^2(\R^3),L_{\alpha}^2(\R^3)} &= \int_{\R^3}\overline{v(y)}\int_{\Omega}\frac{e^{iz|y-y_0|}}{4\pi|y-y_0|} \phi_3(x)dx dy \notag\\
&= \langle \phi_3, \left(R_{-\overline z}v\right)(y_0) \rangle_{L^2(\Omega),L^2(\Omega)}.\label{eq:54}
\end{align}
Therefore, by applying \eqref{eq:46}, \eqref{eq:59} and \eqref{eq:54}, we obtain that \eqref{eq:47} holds with the remainder term satisfying \eqref{eq:121}.

\eqref{a4} We note that for $y\in \Gamma$
\begin{align*}
\left(\left(\Phi_{1/\vep} \circ SL_{\vep z}\right)\phi_4\right)(y) = \left(SL_{\vep z} \phi_4\right)(y_0+1/\vep(y-y_0)) =  \int_{\Gamma}\frac{\vep e^{i z|y-y_0 -\vep(x-y_0)|}}{4\pi|y-y_0-\vep(x-y_0)|} \phi_4(x)d\sigma(x).
\end{align*}
Therefore, by using similar duality arguments as employed in the proof statement \eqref{a3} of this lemma, we readily obtain the assertion of this statement.
\end{proof}

We proceed to prove the following inequality.

\begin{lemma}\label{le:1}
Let $\vep>0$ and $V$ be a bounded closed set of $\overline {\CC_+} \backslash \{0\}$. Given two fixed numbers $\mathcal C_1, \mathcal C_2 \in \R_+$, we have
\begin{align}\label{eq:109}
\left|\mathcal C_1 - z^2 - i\varepsilon {z^3 \mathcal C_2}\right| \ge \frac{\sqrt 2 \mathcal C_1}{4}\vep, \quad \textrm{for}\; \vep\in \left(0, \min\left(\frac{1}{2d_{V,\max}\mathcal C_2}, \frac{1}{d_{V,\min}}\right)\right) \textrm{and}\;z\in V. 
\end{align}
Here, $d_{V,\max}:=\max_{z\in V} |z|$, $d_{V,\min}:= \min_{z\in V} |z|$.
\end{lemma}
\begin{proof}
We first note that $d_{V,\min} > 0$ due to the assumption that $V$ is a bounded closed set of $\overline {\CC_+} \backslash \{0\}$. It is easy to verify that 
\begin{align} \label{eq:107}
\mathcal C_1 - z^2 - i\varepsilon {z^3 \mathcal C_2} = \left(\mathcal C_1 + z \sqrt{1 + i\vep z \mathcal C_2}\right) \left(\mathcal C_1 - z \sqrt{1 + i\vep z \mathcal C_2}\right).
\end{align}
Here, $\textrm{Re}\left(\sqrt{\cdot}\right) > 0$. Since $\vep \in \left(0,\left(2d_{V,max}\mathcal C_2\right)^{-1}\right)$, it follows that 
\begin{align}
& |1 + i\vep z \mathcal C_2| \ge \frac 12, \quad \; \textrm{for}\; z\in V, \label{eq:108} \\
&0 \le \arg{\sqrt{1+i\vep z \mathcal C_2}} \le \frac{\pi}4 - \frac{\arg z }2, \quad \textrm{if}\; z\in V \;\textrm{with}\; \arg z \in \left[0, \frac{\pi}2\right], \label{eq:110}\\
& \frac{\pi}4 - \frac{\arg z}2\le \arg{\sqrt{1+i\vep z \mathcal C_2}} < 0, \quad \textrm{if}\; z\in V \;\textrm{with}\; \arg z\in \left(\frac{\pi}2,\pi\right]. \label{eq:111}
\end{align}
Here, $\arg z$ denotes the angle of the complex number $z$ with respect to the positive real axis in the complex plane.

In the sequel, we distinguish between two cases $z \in V$ with $\arg z \in [0, {\pi}/2]$ and $z \in V$ with $\arg z \in ({\pi}/2, \pi]$ to prove \eqref{eq:109}.

\textbf{Case 1:} $z \in V$ with $\arg z \in [0, {\pi}/2]$. In this case,  by \eqref{eq:110}, we readily obtain 
\begin{align}\label{eq:113}
\left|\textrm{Re}\left(\mathcal C_1 + z \sqrt{1 + i\vep z \mathcal \mathcal C_2}\right)\right| \ge \mathcal C_1, \quad z\in V \;\textrm{with}\; \arg z \in \left[0, \frac{\pi}2\right].
\end{align}
Further, with the aid of \eqref{eq:108} and \eqref{eq:110}, we have
\begin{align}
&\left|\textrm{Im}\left(\mathcal C_1 - z \sqrt{1 + i\vep z \mathcal \mathcal C_2}\right)\right| = \left|\textrm{Im} \left(z \sqrt{1 + i\vep z \mathcal \mathcal C_2}\right) \right| \notag\\
&\qquad \qquad \qquad \qquad \qquad 
\ge \frac{\sqrt 2} 2 |z| \left|\sqrt{1 + i\vep z \mathcal \mathcal C_2}\right| \ge \frac{\sqrt 2} 4 |z|, \quad z\in V \;\textrm{with}\; \arg z \in \left[0, \frac{\pi}2\right]. \label{eq:114}
\end{align}
Combining \eqref{eq:107}, \eqref{eq:113} and \eqref{eq:114} gives that \eqref{eq:109} holds for the case when $z \in V$ with $\arg z \in [0, {\pi}/2]$.

\textbf{Case 2:} $z \in V$ with $\arg z \in ({\pi}/2, \pi]$. In this case, utilizing  \eqref{eq:111} leads to
\begin{align} \label{eq:115}
\left|\textrm{Re}\left(\mathcal C_1 - z \sqrt{1 + i\vep z \mathcal \mathcal C_2}\right)\right| \ge \mathcal C_1, \quad z\in V \;\textrm{with}\; \arg z\in \left(\frac{\pi}2,\pi\right].
\end{align}
Proceeding as in the derivation of \eqref{eq:114}, we can apply \eqref{eq:108} and \eqref{eq:111} to get
\begin{align*}
\left|\textrm{Im}\left(\mathcal C_1 + z \sqrt{1 + i\vep z \mathcal \mathcal C_2}\right)\right| \ge \frac{\sqrt 2} 2 |z| \left|\sqrt{1 + i\vep z \mathcal \mathcal C_2}\right| \ge \frac{\sqrt 2} 4 |z|, \quad z\in V \;\textrm{with}\; \arg z\in \left(\frac{\pi}2,\pi\right].
\end{align*}
This, together with \eqref{eq:115} yields that \eqref{eq:109} holds for the case when $z \in V$ with $\arg z \in ({\pi}/2, \pi]$.
\end{proof}

Now we present an estimate for the operator $R_z$.
\begin{lemma} \label{le:12}
Let $z\in \overline {\CC_+} \backslash \{0\}$ and $\alpha > 1/2$, we have  
\begin{align}
&\|R_z\|_{L_{\alpha}^2(\R^3), H^2_{-\alpha}(\R^3)} \le C \frac{1+|z|^{2}}{|z|}. \label{eq:129}
\end{align}
Here, $C$ is a constant independent of $z$.
\end{lemma}

\begin{proof}
The inequality \eqref{eq:129} directly follows from Proposition 1.2 in \cite{RT}.
\end{proof}

We conclude this section with the introduction of three useful integral identities.

\begin{lemma}\label{le:6}
We have
\begin{align}
&\frac{1}{8\pi}\int_{\Gamma} \int_{\Gamma} \frac{\nu(x)\cdot (x-y)}{|x-y|} \left(S^{-1}_01\right)(y) d\sigma(y)d\sigma(x) = |\Omega|, \label{eq:91}\\
& \frac{1}{8\pi}\int_{\Gamma} \int_{\Gamma} \frac{\nu(y)\cdot (x-y)}{|x-y|} \left(S^{-1}_01\right)(x) d\sigma(y)d\sigma(x) = -|\Omega| \label{eq:92}
\end{align}
and 
\begin{align} \label{eq:93}
\int_{\Gamma}\int_{\Gamma} \nu(x) \cdot (x-y) \left(S^{-1}_01\right)(y) d\sigma(x) d\sigma(y) =3{\mathcal C_\Omega|\Omega|}.
\end{align}
\end{lemma}
\begin{proof}
Firstly, we prove that \eqref{eq:91} holds. By Green formulas, we have
\begin{align*}
\int_{\Gamma} \frac{\nu(x)\cdot (x-y)}{|x-y|}d\sigma(x) =  \int_{\Omega} \Delta |x-y|dx = \int_{\Omega} \frac{2}{|x-y|} dx.
\end{align*}
From this, we get
\begin{align*}
\frac{1}{8\pi}\int_{\Gamma} \int_{\Gamma} \frac{\nu(x)\cdot (x-y)}{|x-y|} \left(S^{-1}_01\right)(y) d\sigma(y)d\sigma(x)  = \int_{\Omega}  \int_{\Gamma} \frac{1}{4\pi|x-y|} \left(S^{-1}_01\right)(y) d\sigma(y)dx,
\end{align*}
whence \eqref{eq:92} follows by the fact that $SL_0 S_0^{-1}$ solves the Laplace equation with the Dirichlet boundary condition of being equal to $1$ on the boundary $\Gamma$.

Secondly, proceeding as in the derivation of \eqref{eq:91}, we can get \eqref{eq:92}.

Thirdly, by employing the identities  
\begin{align*}
\int_{\Gamma} \nu(x) \cdot x d\sigma(x) = \int_{\Omega} \nabla \cdot x dx = 3|\Omega|,  \quad \int_{\Gamma} \nu(x) \cdot 1 d\sigma(x) =0,
\end{align*}
we can directly obtain \eqref{eq:93} holds.

\end{proof}

\section{Global asymptotics of the acoustic field in both space and frequency} \label{section-3}
This section is devoted to proving Theorem \ref{th:1}. We begin with the following observation. Let $w^{in}_{\omega}(y):= u^{in}_{\omega}(y_0+\vep(y-y_0))$ and $w_{\omega,\vep}(y):= u_{\omega,\vep}(y_0+\vep(y-y_0))$ for $y\in \R^3$. Clearly, $w_{\omega}^{in}$ and $w_{\omega,\vep}$ solve 
\begin{align*}
\nabla \cdot \frac{1}{\rho_0} \nabla w_\omega^{in}  + \vep^2\omega^2 \frac{1}{k_0} w^{in}_\omega = 0 \quad \; \textrm{in}\; \R^3
\end{align*}
and 
\begin{align*}
&\nabla \cdot \frac{1}{\rho_\vep \circ \Phi_{\vep}} \nabla w_{\omega,\vep} + \vep^2\omega^2 \frac{1}{k_\vep \circ \Phi_{\vep}} w_{\omega,\vep}= 0  \quad \; \;\;\;\;\;\;\;\; \; \text{in}\; \mathbb R^3,\\
& w_{\omega,\vep} = w^{sc}_{\omega,\vep} + w^{in}_\omega \;\;\;\;\;\; \qquad \qquad \qquad \qquad \qquad \qquad \text{in}\; \mathbb R^3,\\
&\lim_{|x|\rightarrow +\infty}\left(\frac{x}{|x|}\cdot\nabla-i\frac{\vep\omega}{c_0}\right)w^{sc}_{\omega,\vep} = 0,
\end{align*}
respectively. Here, $\Phi_{\vep}$ is given by \eqref{eq:64}. Clearly, $w_{\omega,\vep} \in H^2_{-\alpha}(\R^3 \backslash \Gamma)$. Note that 
\begin{align}\label{eq:52}
u_{\omega,\vep} =w_{\omega,\vep} \circ \Phi_{1/\vep}.
\end{align}
Therefore, in order to investigate the asymptotic behaviors of the field $u_{\omega,\vep}$, it suffices to derive the asymptotic expansion of $w_{\omega,\vep}$. With the aid of integral equations \eqref{eq:17}, \eqref{eq:44} and \eqref{eq:49}, we can find $w_{\omega,\vep}$ that solves 
\begin{align}
& w_{\omega,\vep}(x)= w_\omega^{in}(x) +\left(\frac{1}{c^2_1}- \frac{1}{c_0^2}\right)\vep^2\omega^2\int_{\Omega} \frac{e^{i{\vep\omega}|x-y|/c_0}}{4\pi|x-y|}w_{\omega,\vep}(y)dy\notag\\
&\quad\quad \quad \;-\left(\frac{\rho_0}{\rho_1\vep^2}-1\right)\int_{\Gamma}\frac{e^{i{\vep\omega}{|x-y|}/c_0}}{4\pi|x-y|} \partial_\nu w_{\omega,\vep}(y)d\sigma(y), \quad x\in \R^3 \backslash \Gamma, \label{eq:62}
\end{align} 
where the value $w_{\omega,\vep}$ within $\Omega$ and the normal derivative $\partial_\nu w_{\omega, \vep}$ on $\Gamma$ are determined by
\begin{align}\label{eq:18}
\left(\mathbb I-\left(\frac{1}{c^2_1}- \frac{1}{c_0^2}\right)\vep^2\omega^2N_{\vep\omega/c_0}\right)w_{\omega,\vep}&= w_\omega^{in} -\left(\frac{\rho_0}{\rho_1\vep^2}-1\right)SL_{\vep\omega/c_0} \partial_\nu w_{\omega,\vep} \quad \textrm{in}\; \Omega
\end{align}
and 
\begin{align}
&\frac{\rho_0}{\rho_1\vep^2}\left(\frac 12\left(1 + \frac {\rho_1\vep^2} {\rho_0}\right) \mathbb I + \left(1-\frac{\rho_1\vep^2} {\rho_0}\right)K_{\vep\omega/c_0}^*\right) \partial_\nu w_{\omega,\vep}  \notag \\
&\qquad\qquad\qquad\qquad\qquad\qquad\qquad = \partial_\nu w_\omega^{in} + \left(\frac{1}{c^2_1}-\frac{1}{c^2_0}\right)\vep^2\omega^2\partial_\nu N_{\vep\omega/c_0} w_{\omega,\vep} \quad\textrm{on}\; \Gamma. \label{eq:16}
\end{align}
For every $z \in \overline{\CC_+} \backslash \{0\}$, define
\begin{align}
\Lambda^{(1)}_{z}:= \mathbb I-\left(\frac{1}{c^2_1}- \frac{1}{c_0^2}\right)z^2N_{z/c_0}, \quad \Lambda^{(2)}_{z,\vep}:=\frac 12\left(1 + \frac{\rho_1 \vep^2}{\rho_0}\right) \mathbb I + \left(1-\frac{\rho_1 \vep^2}{\rho_0}\right)K_{z/c_0}^*. \label{eq:23} 
\end{align}
Based on integral equations \eqref{eq:62}, \eqref{eq:18} and \eqref{eq:16}, obtaining asymptotic estimates of the inverse of the operators $\Lambda^{(1)}_{\vep\omega}$ and $\Lambda^{(2)}_{\vep\omega,\vep}$ as $\vep$ tends to $0$ plays an essential role in deriving asymptotic expansions of the field $w_{\omega,\vep}$. It is readily observed that $\Lambda^{(1)}_{\vep\omega}\approx \mathbb I$ when $\vep$ tends to $0$, leading to its inverse also approximately scaling as $\left(\Lambda^{(1)}_{\vep\omega}\right)^{-1}\approx \mathbb I$ for sufficiently small $\vep$. Similarly, $\Lambda^{(2)}_{\vep\omega,\vep}$ can be expected to approximate $1/2 \mathbb I + K^*_0 $ as $\vep$ approaches to $0$. However, $-1/2$ is the eigenvalue of the operator $K^*_0$, which poses challenges in estimating the inverse of  $\Lambda^{(2)}_{\vep\omega,\vep}$ for sufficiently small $\vep$. To overcome this difficulty, we utilize the spectral properties of  $K^*_0$.

Building on the preceding discussions, we introduce the spectral properties of the operator $K^*_0$ in the subsequent subsection before proceeding to prove Theorem \ref{th:1}.   
\subsection{\texorpdfstring {Spectral properties of $K^*_0$}{}}

We begin by outlining the following important spectral properties of $K^*_0$.

\begin{lemma}
$K_0^*$ is a compact operator of $H^{-1/2}(\Gamma)$ and $\lambda_0 = -1/2$ is a simple eigenvalue of the operator $K^*_0$ and the corresponding eigenvalue function is $\left(S^{-1}_01\right)(x)$.
\end{lemma}

For any $\phi \in H^{-1/2}(\Gamma)$, we define
\begin{align}\label{eq:71}
\left(\mathcal P \phi\right)(x):= \mathcal C^{-1}_\Omega\int_{\Gamma} \left(S_0 \phi\right)(y) \left(S^{-1}_01\right)(y) d\sigma(y) (S^{-1}_01)(x), \quad x\in \Gamma.
\end{align}
Clearly, the operator $\mathcal P$ projects $\phi$ onto the eigenspace of the operator $K^*_0$ corresponding to the eigenvalue $-1/2$, which is spanned by $\left(S^{-1}_01\right)(x)$ and is denoted by $\textrm{Span}\{S^{-1}_01\}$. Define a novel scalar product
\begin{align}\label{eq:83}
\langle \phi, \psi \rangle_{S_0}:=\mathcal C^{-1}_\Omega\int_{\Gamma} \left(S_0 \phi\right)(y)\psi(y)d\sigma(y), \quad \phi,\psi\in H^{-\frac12}(\Gamma).
\end{align}
This scalar product $\langle \cdot, \cdot \rangle_{S_0}$ is well defined since $S_0 \in \mathcal L(H^{-1/2}(\Gamma), H^{1/2}(\Gamma))$. The constant $\mathcal C^{-1}_\Omega$, as specified in \eqref{eq:74}, ensures that $\langle S^{-1}_01,  S^{-1}_01\rangle_{S_0} = 1$. By \eqref{eq:71}, we readily find
\begin{align} 
\mathcal P \phi = \langle \phi, S^{-1}_01\rangle_{S_0} S^{-1}_01, \quad \phi -\mathcal P \phi\in \textrm{Span}\{S^{-1}_01\}^{\perp}. \label{eq:140}
\end{align}
Here, $\textrm{Span}\{ S^{-1}_01\}^{\perp}:=\{\phi\in H^{-1/2}(\Gamma):\langle \phi,S^{-1}_01\rangle_{S_0} =0\}$. By the definition of $\textrm{Span}\{ S^{-1}_01\}^{\perp}$, it is readily deduced that
\begin{align} \label{eq:77}
\left(\frac 12 \mathbb I+K^*_0\right)^{-1} \in \mathcal L\left(\textrm{Span}\{ S^{-1}_01\}^{\perp}\right).
\end{align}
Further, we note that every $\phi \in H^{-1/2}(\Gamma)$ can be decomposed into 
\begin{align*}
\phi = \mathcal P \phi + \left(\mathbb I-\mathcal P\right)\phi =: a_{\phi} S^{-1}_01 + \phi_r.
\end{align*}
Here, $a_\phi \in \mathbb C$ denotes the projection coefficient preceding the eigenvector $S^{-1}_01$ and $\phi_r$ belongs to $\textrm{Span}\{ S^{-1}_01\}^{\perp}$. Based on this, every operator $H\in \mathcal L(H^{-1/2}(\Gamma))$ can be represented by
\begin{align}\label{eq:72}
\left(H \phi\right) (x)  = a^H_{\phi} \left(S^{-1}_01\right)(x) + \phi^H_r(x), \quad  x\in \Gamma,
\end{align}
where $a_\phi$, $\phi_r$, $a^H_{\phi}$ and $\phi^H_r$ satisfy
\begin{align} \label{eq:73}
\begin{pmatrix}
a^H_{\phi}\\
\phi^H_r
\end{pmatrix} =  \begin{bmatrix}
    H_{00} & H_{01}\\
    H_{10} & H_{11}
\end{bmatrix}
\begin{pmatrix}
a_{\phi}\\
\phi_r
\end{pmatrix}.
\end{align}
Here, $H_{00}$ is a complex number, $H_{01} \in \mathcal L (\textrm{Span}\{ S^{-1}_01\}^{\perp}, \CC)$, $H_{10} \in \mathcal L(\CC,\textrm{Span}\{ S^{-1}_01\}^{\perp})$, and $H_{11} \in \mathcal L(\textrm{Span}\{ S^{-1}_01\}^{\perp})$. The upcoming theorem will provide a characterization of the inverse for a class of operators based on the above representation \eqref{eq:72}. 

\begin{lemma}\label{le:3}
Let $H \in \mathcal L(H^{-1/2}(\Gamma))$ be defined as in \eqref{eq:72}, with $a^H_\phi$ and $\phi^H_r$ are determined by \eqref{eq:73}.
Suppose that $H_{11}$ has a bounded inverse $H^{-1}_{11}\in\mathcal L(\textrm{Span}\{ S^{-1}_01\}^{\perp})$, and that $H_{00}-H_{01}H^{-1}_{11}H_{10}{1} \ne 0$. Then $H^{-1}$ exists, represented by
\begin{align}\label{eq:57}
H^{-1}\phi = \frac{a_\phi - H_{01}H^{-1}_{11}\phi_r}{H_{00}-H_{01}H^{-1}_{11}H_{10}1} S^{-1}_01+ \left[-\frac{a_\phi - H_{01}H^{-1}_{11}\phi_r}{H_{00}-H_{01}H^{-1}_{11}H_{10}1}H_{11}^{-1}H_{10}{1} + H^{-1}_{11}\phi_r\right].
\end{align}
\end{lemma}

\begin{proof}
Given $\phi\in H^{-1/2}(\Gamma)$, we aim to find the solution of 
\begin{align*}
H f = a^H_{f} S^{-1}_01 + f^H_r =(H_{00} a_f + H_{01}f_r) S^{-1}_01 + H_{10} a_f + H_{11}f_r =\phi = a_{\phi} S^{-1}_01 + \phi_r.
\end{align*}
This is equivalent to solve
\begin{align*}
H_{00} a_f + H_{01}f_r = a_\phi, \\
H_{10} a_f + H_{11}f_r = \phi_r.
\end{align*}
Based on the assumptions that $H_{00}-H_{01}H^{-1}_{11}H_{10}{1} \ne 0$ and $H_{11}$ has a bounded inverse $H^{-1}_{11}\in\mathcal L(\textrm{Span}\{ S^{-1}_01\}^{\perp})$, a straightforward calculation gives 
\begin{align*}
a_f = \frac{a_\phi - H_{01}H^{-1}_{11}\phi_r}{H_{00}-H_{01}H^{-1}_{11}H_{10}1}, \quad f_r = -\frac{a_\phi - H_{01}H^{-1}_{11}\phi_r}{H_{00}-H_{01}H^{-1}_{11}H_{10}1}H_{11}^{-1}H_{10}{1} + H^{-1}_{11}\phi_r.
\end{align*}
Therefore, we conclude that $H^{-1}$ exists and is given explicitly by \eqref{eq:57}.
\end{proof}

Utilizing the representation \eqref{eq:72} offers the advantage of estimating the inverse of operators. Notably, the Born series inversion method, widely utilized for estimating inverses of operators as in \cite{MPS}, requires that the operator can be expressed as a sum of the identity operator and another operator with a norm less than $1$. In contrast, our novel representation simplifies the task, only requiring the estimation of the inverse of the projection coefficient associated with the eigenvector, thereby bypassing the stringent assumptions required by the Born series technique. Employing this approach to estimate the inverse of the operator class $\Lambda^{(2)}_{\vep z,\vep} +\vep^2\beta \mathcal P $ brings us to the following lemma, where $\beta \in \R \backslash \R_-$, $\Lambda^{(2)}_{\vep z,\vep}$ and $\mathcal P$ are specified in \eqref{eq:23} and  \eqref{eq:71}, respectively.

\begin{lemma} \label{le:2}
Let $\vep>0$ and $\beta \in \R \backslash \R_-$. Assume that $V$ be a bounded closed set of $\overline {\CC_+} \backslash \{0\}$. There exists $\delta_V \in \R_+$ such that for any $\phi\in H^{-1/2}(\Gamma)$, we have 
\begin{align}\label{eq:67}
\vep^2\left(\left(\Lambda^{(2)}_{\vep z,\vep}+ \vep^2 \beta \mathcal P \right)^{-1} \phi\right) (x) = \frac{\langle \phi,  S^{-1}_01\rangle_{S_0}}{\frac{\rho_1}{\rho_0} + \beta -\frac{z^2 |\Omega|}{\mathcal C_\Omega c^2_0}-i\frac{z^3|\Omega|}{4\pi c^3_0}\vep}\left(S^{-1}_01\right)(x) + (r_{Res}\phi)(x), \;\; x\in \Gamma,
\end{align}
where 
\begin{align}\label{eq:68}
\|r_{Res}(\phi)\|_{{H^{-\frac 12}(\Gamma)}} \le C_{d_{V,\max}}\frac{\vep^2 \left|\langle \phi, S^{-1}_01\rangle_{S_0}\right| + \vep^2\|\phi - \mathcal P \phi\|_{H^{-\frac12}(\Gamma)} }{\left|\frac{\rho_1}{\rho_0} + \beta -\frac{z^2 |\Omega|}{\mathcal C_\Omega c^2_0}-i\frac{z^3|\Omega|}{4\pi c^3_0}\vep\right|}  
\end{align}
holds uniformly with respect to all $z \in V$ and all $\vep \in (0,\delta_V)$. Here, $d_{V,\max}:=\max_{z\in V} |z|$ and the positive constant $C_{ d_{V,\max}}$ is independent of $\vep$ and $z$.
\end{lemma}

\begin{proof}
Assume that $\vep <1$ throughout the proof. Define $\mathcal Q\phi= \phi-\mathcal P\phi$ for $\phi \in H^{-1/2}(\Gamma)$. It follows from statement \eqref{b2} of Lemma \ref{le:5} that 
\begin{align*}
\Lambda^{(2)}_{\vep z,\vep} & = \frac{\rho_1\vep^2}{\rho_0} \mathbb I + \left(1-\frac{\rho_1\vep^2}{\rho_0}\right)\left(\frac{1}2 \mathbb I +K^*_{\vep z/c_0}\right)\\
&=\frac{\rho_1\vep^2}{\rho_0} \mathbb I + \left(1-\frac{\rho_1\vep^2 }{\rho_0}\right)\left(\frac 12 \mathbb I+ K^*_0 + \frac{\vep^2 z^2}{ c^2_0}K^{*,(2)} + \frac{\vep^3 z^3}{c^3_0} K^{*, (3)} + \mathcal R_\Lambda\right),
\end{align*}
where $\|R_\Lambda\|_{H^{-1/2}(\Gamma),H^{-1/2}(\Gamma)} \le C\vep^4 |z|^4$. Clearly,
\begin{align*}
\|\mathcal A R_\Lambda (\mathbb I-\mathcal A)\|_{H^{-1/2}(\Gamma),H^{-1/2}(\Gamma)}\le C_{d_{V,max}}\vep^4|z|^2, \quad \mathcal A = \mathcal P, \mathcal Q.
\end{align*}
This, together with the identities $(\mathcal P +\mathcal Q) = \mathbb I$, $(1/2 \mathbb I+K^*_0)\mathcal P = 0$ and $\mathcal P (1/2 \mathbb I+K^*_0)\mathcal Q = 0$ yields that 
\begin{align*}
&\left(\Lambda^{(2)}_{\vep z,\vep} + \vep^2\beta \mathcal P\right) \phi = (\mathcal P +\mathcal Q) \left(\Lambda^{(2)}_{\vep z,\vep} + \vep^2\beta \mathcal P \right) (\mathcal P +\mathcal Q)\phi = (\mathcal P +\mathcal Q) \bigg[\frac{\rho_1\vep^2}{\rho_0} \mathbb I + \bigg(1-\frac{\rho_1\vep^2 }{\rho_0}\bigg)\\
&\left(\frac 12 \mathbb I+ K^*_0 + \frac{\vep^2 z^2}{ c^2_0}K^{*,(2)} + \frac{\vep^3 z^3}{c^3_0} K^{*, (3)}  \right) + \vep^2\beta \mathcal P + \left(1-\frac{\rho_1\vep^2 }{\rho_0}\right)R_{\Lambda}\bigg](\mathcal P +\mathcal Q)\phi\\
&= \left[M_{00}a_\phi + M_{01}\phi_r\right]S^{-1}_01+ \left[M_{10} a_\phi+ M_{11} \phi_r\right],
\end{align*}
for every $\phi= a_\phi S^{-1}_01+ \phi_r \in H^{-1/2}(\Gamma)$ with $a_\phi = \langle \phi, S^{-1}_01\rangle_{S_0}$ and $\phi_r \in \textrm{Span}\{S^{-1}_01\}^{\perp}$, where $M_{00}$, $M_{01}$, $M_{10}$ and $M_{11}$ satisfy
\begin{align}
&M_{00} \in \CC, \quad \left|M_{00}-\frac{\rho_1\vep^2}{\rho_0}-\vep^2\beta-\frac{\vep^2z^2}{c^2_0} \langle K^{*,(2)} S^{-1}_01, S^{-1}_01\rangle_{S_0} - \frac{\vep^3 z^3}{c^3_0}\langle K^{*,(3)}S^{-1}_01, S^{-1}_01\rangle_{S_0}\right|\notag\\ 
&\qquad \qquad \qquad \qquad \qquad \qquad \qquad \qquad \qquad \qquad \qquad \qquad \qquad \qquad \quad\le C_{d_{V,max}}\vep^4 |z|^2, \label{eq:20}\\
&M_{01}\in \mathcal L (\textrm{Span}\{S^{-1}_01\}^{\perp}, \CC),\quad \left\|M_{01}\right\|_{H^{-\frac12}(\Gamma), \CC} \le C_{d_{V,max}}\vep^2|z|^2,  \label{eq:21}
\end{align}
\begin{align}
&M_{10}\in \mathcal L (\CC, \textrm{Span}\{ S^{-1}_01\}^{\perp}), \quad \left\|M_{10}\right\|_{\CC,H^{-\frac12}(\Gamma)} \le C_{d_{V,max}}\vep^2|z|^2, \label{eq:65} \\
&M_{11}\in \mathcal L (\textrm{Span}\{ S^{-1}_01\}^{\perp}), \quad \left\|M_{11} - Q_0\left(\frac 12 \mathbb I+ K^*_0\right)Q_0\right\|_{H^{-\frac12}(\Gamma),H^{-\frac12}(\Gamma)}\notag\\
&\qquad \qquad \qquad \qquad \qquad \qquad \qquad \qquad \qquad \qquad \qquad \qquad\qquad\quad\le C_{d_{V,max}}(1+|z|^2)\vep^2.  \label{eq:19}
\end{align}
Furthermore, by Lemma \ref{le:6}, we have 
\begin{align*}
\frac{\rho_1\vep^2}{\rho_0} + \beta \vep^2& + \frac{\vep^2 z^2}{c^2_0} \langle K^{*,(2)} S^{-1}_01, S^{-1}_01\rangle_{S_0} + \frac{\vep^3 z^3}{c^3_0}\langle K^{*,(3)}S^{-1}_01, S^{-1}_01\rangle_{S_0} \\
&= \vep^2\left(\frac{\rho_1}{\rho_0} + \beta - \frac{z^2}{\mathcal C_\Omega c_0^2}|\Omega|-\frac{i\vep z^3 \mathcal|\Omega|}{4\pi c^3_0}\right).
\end{align*}
From this, we can employ \eqref{eq:20}, \eqref{eq:21} and \eqref{eq:65} to get that there exists $\delta_V^{(1)} \in \R_+$ such that 
\begin{align}
&M_{00}- M_{01}M^{-1}_{11}M_{10}1 \ne 0, \quad \mathrm{and}\notag\\
&\left|M_{00}- M_{01}M^{-1}_{11}M_{10}1 - \vep^2\left(\frac{\rho_1}{\rho_0} + \beta - \frac{z^2}{\mathcal C_\Omega c_0^2}|\Omega|-\frac{i\vep z^3 \mathcal|\Omega|}{4\pi c^3_0}\right) \right| \le C_{d_{V,\max}}\vep^4|z|^2 \label{eq:56}
\end{align}
for all $\vep \in \left(0, \delta_V^{(1)}\right)$. Moreover, we can deduce from \eqref{eq:77} and \eqref{eq:19} that there exists $\delta_V^{(2)} \in \R_+$ such that when $\vep \in \left(0, \delta^{(2)}_V\right)$, $M_{11}$ has an inverse $M^{-1}_{11}\in\mathcal L(\textrm{Span}\{ S^{-1}_01\}^{\perp})$ and 
\begin{align} \label{eq:55}
\|M^{-1}_{11}\|_{\mathcal L(\textrm{Span}\{ S^{-1}_01\}^{\perp})} \le C_{d_{V,\max}}.
\end{align}
Based on the above discussions, we can utilize Lemma \ref{le:3} to get 
\begin{align}
\left(\Lambda^{(2)}_{\vep z,\vep} + \vep^2 \beta \mathcal P \right)^{-1}\phi & = \frac{a_\phi- M_{01}M_{11}^{-1}\phi_r}{M_{00}-M_{01}M^{-1}_{11}M_{10}1}\notag\\
&+\left[-\left(\frac{a_\phi-M_{01}M_{11}^{-1}\phi_r}{M_{00} -M_{01}M^{-1}_{11}M_{10}1}M_{11}^{-1}M_{10}{1}\right) + M_{11}^{-1}\phi_r\right].\label{eq:66}
\end{align}
We set $\delta_V:=\min\left(1,{2\pi c^2_0/\left(d_{V,\max}|\Omega|\right)}, {1/d_{V,\min}}, \delta_V^{(1)}, \delta_V^{(2)}\right)$, where $d_{V,\min}:= \min_{z\in V} |z|$. It follows from \eqref{eq:56} and Lemma \ref{le:1} that 
\begin{align*}
\frac{\vep^2}{M_{00} -M_{01}M^{-1}_{11}M_{10}1} - \frac{1}{\left|\frac{\rho_1}{\rho_0} + \beta -\frac{z^2 |\Omega|}{\mathcal C_\Omega c^2_0}-i\frac{z^3|\Omega|}{4\pi c^3_0}\vep\right|} \le \frac{C_{d_{V,\max}}\vep}{\left|\frac{\rho_1}{\rho_0} + \beta -\frac{z^2 |\Omega|}{\mathcal C_\Omega c^2_0}-i\frac{z^3|\Omega|}{4\pi c^3_0}\vep\right|}.
\end{align*}
This, together with \eqref{eq:21}, \eqref{eq:65}, \eqref{eq:19}, \eqref{eq:55}, \eqref{eq:66} and the fact that $a_\phi = \langle \phi, S^{-1}_01\rangle_{S_0}$ shows that the operator $\vep^2 \left(\Lambda^{(2)}_{\vep z,\vep} + \vep^2 \beta \mathcal P\right)^{-1}$ has the asymptotic expansion \eqref{eq:67} with the remainder term $r_{Res}(\phi)$ satisfying \eqref{eq:68} for all $\vep \in (0,\delta_V)$. The proof of this lemma is thus completed.

\end{proof}

Now we are in a position to give the proof of Theorem \ref{th:1}. We begin by proving Theorem \ref{th:1} for the simpler case when $c_1 = c_0$ in section \ref{sec:3.1}, which will provide a clear understanding to the main idea of the proof. Subsequently, building upon the approach used to prove the case $c_1=c_0$, we will extend our proof to the more general case when $c_1 \ne c_0$ in section \ref{sec:3.2}. 

\subsection{\texorpdfstring {Proof of Theorem \ref{th:1} for the case $c_1 = c_0$}{}} \label{sec:3.1}

\begin{proof}[Proof of Theorem \ref{th:1} for the case $c_1 = c_0$]
Let $\vep>0$ be sufficiently small throughout the proof. As $c_1 = c_0$, it is easily seen from \eqref{eq:52} and \eqref{eq:62} that
\begin{align}\label{eq:7}
u_{\omega,\vep} = u_{\omega}^{in} - \left(\frac{\rho_0}{\rho_1\vep^2} - 1\right) \left(\Phi_{1/\vep} \circ SL_{\vep\omega/c_0}\right) \partial_\nu w_{\omega,\vep}, \quad \textrm{in}\; \R^3 \backslash \Gamma.
\end{align}
First, we focus on the estimate of $\partial_\nu w_{\omega,\vep}$. By \eqref{eq:16}, we deduce
\begin{align} \label{eq:95}
\partial_\nu w_{\omega,\vep}  = \frac{\rho_1\vep^2}{\rho_0} \left(\Lambda_{\vep\omega,\vep}^{(2)}\right)^{-1}\partial_\nu w_\omega^{in} \quad \textrm{on}\; \Gamma.
\end{align}
It should be noted that, according to Lemma \ref{le:2}, the inverse of $\Lambda_{\vep\omega,\vep}^{(2)}$ exists. For the estimate of $\partial_\nu w_\omega^{in}$ on $\Gamma$, given that the field $w_\omega^{in}$ solves the Helmholtz equation with the wave number $\vep^2\omega^2/c^2_0$ in $\Omega$, and given that $\vep$ is small enough such that $\vep^2\omega^2/c^2_0$ is not a Dirichlet eigenvalue of $-\Delta$ in $\Omega$, we can deduce from \eqref{eq:51} that
\begin{align} \label{eq:22}
\partial_\nu w_\omega^{in} =  S_{\vep\omega/c_0}^{-1}\left (\frac 12 \mathbb I+ K_{\vep\omega/c_0}\right) \gamma w^{in}_\omega \quad \textrm{on}\; \Gamma.
\end{align}
With the aid of statement \eqref{a1} of Lemma \ref{le:7}, we have
\begin{align} \label{eq:85} 
& \|\gamma\left(w^{in}_{\omega} - u_{\omega}^{in}(y_0)\right)\|_{H^{\frac 12}(\Gamma)} \le C\vep^{\frac 12}\|u_{\omega}^{in}\|_{H^2(B_1(y_0))}.
\end{align}
This, together with \eqref{eq:122}, \eqref{eq:22} and Lemma \ref{le:8} gives 
\begin{align}\label{eq:70}
\left|\langle \partial_\nu w_\omega^{in}, S^{-1}_01 \rangle_{S_0} - \frac{\vep^2 w^2}{c^{2}_0} \langle S_0^{-1} K^{(2)}1, S^{-1}_01\rangle_{S_0} u_{\omega}^{in}(y_0)\right| \le C_{d_{I,\max}}\vep^{\frac 52}.
\end{align}
Furthermore, using \eqref{eq:78}, \eqref{eq:79} and the fact that $(1/2 \mathbb I + K_0) 1 = 0$, we have
\begin{align}\label{eq:112}
\left\| S_{\vep\omega/c_0}^{-1}\left (\frac 12 \mathbb I+ K_{\vep\omega/c_0}\right) 1 \right\|_{H^{-\frac 12}(\Gamma)} \le C_{d_{I,\max}}\vep^2.
\end{align}
Combining \eqref{eq:85}, \eqref{eq:70}, \eqref{eq:112}, Lemma \ref{le:6} and Lemma \ref{le:2} gives 
\begin{align*}
\vep^2\partial_\nu w_{\omega,\vep} = \frac{\rho_1\vep^2}{\rho_0} \left[\frac{-{\vep^2\omega^2}{c_0^{-2}}|\Omega|\mathcal C^{-1}_\Omega u_\omega^{in}(y_0)}{\frac{\rho_1}{\rho_0}-\frac{\omega^2 |\Omega|}{\mathcal C_\Omega c^2_0}-i\frac{\omega^3|\Omega|}{4\pi c^3_0}\vep} S^{-1}_01  + Res\right]\;\; \textrm{on}\; \Gamma,
\end{align*}
where $Res$ satisfies
\begin{align*}
\left\|Res\right\|_{{H^{-\frac 12}(\Gamma)}} \le \frac{C_{d_{I,\max}}\vep^{\frac 52}}{\left|\frac{\rho_1}{\rho_0} -\frac{\omega^2|\Omega|}{\mathcal C_\Omega c^2_0} - i\frac{\omega^3|\Omega|}{4\pi c^3_0}\vep\right|}.
\end{align*}
Moreover, using statement \eqref{a4} of Lemma \ref{le:7} and Lemma \ref{le:12}, we have
\begin{align*}
& \left\|\left(\Phi_{1/\vep} \circ SL_{\vep\omega/c_0}\right) S^{-1}_01 - \vep \mathcal C_\Omega \frac{e^{i\omega|\cdot-y_0|/c_0}}{4\pi|\cdot-y_0|}\right\|_{L^2_{-\alpha}(\R^3)} \le C_{d_{I,\max}, d_{I,\min}}\vep^{\frac 32},\\
& \left\|\left(\Phi_{1/\vep} \circ SL_{\vep\omega/c_0}\right) Res - \vep \mathcal C_\Omega \int_{\Gamma} Res(y)d\sigma(y) \frac{e^{i\omega|\cdot-y_0|/c_0}}{4\pi|\cdot-y_0|}\right\|_{L^2_{-\alpha}(\R^3)} \le \frac{C_{d_{I,\max},d_{I,\min}}\vep^{4}}{\left|\frac{\rho_1}{\rho_0} -\frac{\omega^2|\Omega|}{\mathcal C_\Omega c^2_0} - i\frac{\omega^3|\Omega|}{4\pi c^3_0}\vep\right|}.
\end{align*}  
From this, utilizing \eqref{eq:7}, \eqref{eq:95}, and the estimate of $\partial_\nu w_\omega^{in}$ on $\Gamma$ yields that when $c_1=c_0$, $u_{\omega,\vep}$ has the asymptotic expansion \eqref{eq:4} with the remainder term $u^{res}_{\omega,\vep}$ satisfying \eqref{eq:8} uniformly with respect to all $\omega \in I$.
\end{proof}

\subsection{\texorpdfstring{Proof of Theorem \ref{th:1} for the case $c_1\ne c_0$}{}} \label{sec:3.2}

The proof of Theorem \ref{th:1} for the case $c_1 \ne c_0$ is similar to that of the case $c_1 = c_0$. However, the integral representations \eqref{eq:62}, \eqref{eq:18} and \eqref{eq:16} for $c_1\ne c_0$ are significantly more complex than those for the case $c_1=c_0$. To address this, we require the following new identity.

\begin{lemma} \label{le:9}
Let $\vep>0$ and $w_{\omega,\vep}$ be the solution of \eqref{eq:62}. We have
\begin{align}\label{eq:86}
\langle \partial_\nu N_{\vep\omega/c_0} w_{\omega,\vep}, S^{-1}_01 \rangle_{S_0} = -\frac{\vep^2\omega^2}{\mathcal C_\Omega c^2_0}\int_\Omega \int_\Omega \frac{e^{i{\vep\omega}|x-y|/c_0}}{4\pi|x-y|} w_{\omega,\vep}(y)dxdy + \frac{c^2_1}{\vep^2\omega^2}\left\langle \partial_\nu w_{\omega,\vep}, S^{-1}_01 \right\rangle_{S_0}.
\end{align}
\end{lemma}

\begin{proof}
By the definition of the scalar product $\langle \cdot,\cdot \rangle_{S_0}$ specified in \eqref{eq:83} and the fact that $w_{\omega,\vep} \in H^2(\Omega)$, we easily find 
\begin{align}
\mathcal C_\Omega\left\langle \partial_\nu N_{\vep\omega/c_0} w_{\omega,\vep}, S^{-1}_01 \right\rangle_{S_0} &= \int_\Gamma \partial_{\nu_x} \int_{\Omega} \frac{e^{i{\vep\omega}|x-y|/c_0}}{4\pi|x-y|}w_{\omega,\vep}(y)dy d\sigma(x) \notag\\
& = -\frac{\vep^2\omega^2}{c^2_0}\int_\Omega \int_\Omega  \frac{e^{i{\vep\omega}|x-y|/c_0}}{4\pi|x-y|} w_{\omega,\vep}(y) dx dy - \int_\Omega w_{\omega,\vep}(y) dy. \label{eq:84}
\end{align}
Since $w_{\omega,\vep}$ solves the Helmholtz equation with the wave number $\vep^2\omega^2/c^2_1$ in $\Omega$, we have
\begin{align*}
-\int_\Omega w_{\omega,\vep}(y) dy = \frac{c^2_1}{\vep^2\omega^2}\int_\Omega \Delta w_{\omega,\vep}(y) dy= \frac{c^2_1}{\vep^2 \omega^2} \int_\Gamma \partial_\nu w_{\omega,\vep}(y)d\sigma(y) = \frac{c^2_1\mathcal C_\Omega}{\vep^2\omega^2}\left\langle \partial_\nu w_{\omega,\vep}, S^{-1}_01 \right\rangle_{S_0}.
\end{align*}
Combining this with \eqref{eq:84} gives \eqref{eq:86}.
\end{proof}

We are ready to give the proof Theorem \ref{th:1} for the case $c_1\ne c_0$.

\begin{proof}[Proof of Theorem \ref{th:1} for the case $c_1\ne c_0$]
Let $\vep > 0$ be sufficiently small throughout the proof. Similar to the derivation of \eqref{eq:7}, we can use \eqref{eq:52}, \eqref{eq:62} to get
\begin{align} 
u_{\omega,\vep} & = u_{\omega}^{in}-\left(\frac{\rho_0}{\rho_1\vep^2}-1\right) \left(\Phi_{1/\vep} \circ SL_{\vep\omega/c_0}\right)\partial_\nu w_{\omega,\vep} \notag\\
&+ \left(\frac{1}{c^2_1}-\frac{1}{c_0^2}\right){\vep^2\omega^2} \left(\Phi_{1/\vep} \circ N_{\vep\omega/c_0}\right) w_{\omega,\vep}, \quad \textrm{in}\; \R^3 \backslash \Gamma.\label{eq:28}
\end{align}  
In contrast to the case of $c_1 = c_0$, we need to estimate both $\partial_\nu w_{\omega,\vep}$ on $\Gamma$ and $w_{\omega,\vep}$ in $\Omega$.

We first estimate $\partial_\nu w_{\omega,\vep}$ on $\Gamma$. Subtracting $(1-c^2_1/c^2_0) \mathcal P \partial_\nu w_{\omega,\vep}$ on both sides of \eqref{eq:16}, we have
\begin{align*}
\frac{\rho_0}{\rho_1\vep^2}\left(\Lambda_{\vep\omega,\vep}^{(2)} + \vep^2 \zeta \mathcal P\right)\partial_\nu w_{\omega,\vep} & = \partial_\nu w^{in}_\omega + \left(\frac{1}{c^2_1}-\frac{1}{c^2_0}\right){\vep^2\omega^2}\partial_\nu N_{\vep\omega/c_0} w_{\omega,\vep} + \left(\frac{c^2_1}{c^2_0}-1\right) \mathcal P \partial_\nu w_{\omega,\vep}\\\
&=: q_{\omega,\vep}\quad \textrm{on}\; \Gamma,
\end{align*}
where $\zeta:=\rho_1\left(c^2_1/c^2_0-1\right)/\rho_0$. Since the inverse of $\Lambda_{\vep\omega,\vep}^{(2)} + \vep^2 \zeta \mathcal P$ exists by Lemma \ref{le:2}, we have
\begin{align*}
\partial_\nu w_{\omega,\vep} = \frac{\rho_1\vep^2}{\rho_0}\left(\Lambda_{\vep\omega,\vep}^{(2)} + \vep^2 \zeta \mathcal P \right)^{-1} q_{\omega,\vep} \quad \text{on}\;\Gamma.
\end{align*}
Thus, in view of Lemma \ref{le:2}, to derive the estimate of $\partial_\nu w_{\omega,\vep}$ on $\Gamma$, it is necessary to estimate the projection coefficients $\langle q_{\omega,\vep}, S^{-1}_01\rangle_{S_0}$ of $q_{\omega,\vep}$ preceding the function $S^{-1}_01$. 
By the definition of the operator $\mathcal P$, we easily derive
\begin{align*}
\left\langle \mathcal P \partial_\nu w_{\omega,\vep}, S^{-1}_01\right\rangle_{S_0} = \langle \partial_\nu w_{\omega,\vep}, S^{-1}_01\rangle_{S_0}.
\end{align*}
From this, employing \eqref{eq:116} and Lemma \ref{le:9} gives
\begin{align}\label{eq:76}
\left|\langle q_{\omega,\vep}-\partial_\nu w_\omega^{in}, S^{-1}_01 \rangle_{S_0}\right|\le C_{d_{I,\max}}\vep^4\|N_{\vep\omega/c_0}\|_{L^2{(\Omega)}, H^2(\Omega)} \|w_{\omega,\vep}\|_{L^2(\Omega)}.
\end{align}
Combining \eqref{eq:116}, \eqref{eq:70} and \eqref{eq:76} yields 
\begin{align}\label{eq:24}
\left|\langle q_{\omega,\vep}, S^{-1}_01 \rangle_{S_0} - \frac{\vep^2 w^2}{c^{2}_0} \langle S_0^{-1} K^{(2)}1, S^{-1}_01\rangle_{S_0} u_{\omega}^{in}(y_0)\right| \le C_{d_{I,\max}}\left(\vep^{\frac 52} + \vep^4\|w_{\omega,\vep}\|_{L^2(\Omega)}\right).
\end{align}
Furthermore, using \eqref{eq:116} again and applying the trace formula  $\|\partial_\nu\phi\|_{H^{-1/2}(\Gamma)} \le C \|\phi\|_{H^1(\Omega)}$ for any $\phi \in H^2(\Omega)$, we find
\begin{align*}
\left\|{\vep^2\omega^2}\partial_\nu N_{\vep\omega/c_0} w_{\omega,\vep}\right\|_{H^{-\frac 12}(\Gamma)} \le C_{d_{I,\max}}\vep^2\|w_{\omega,\vep}\|_{L^2(\Omega)}.
\end{align*}
Therefore, using \eqref{eq:85}, \eqref{eq:112}, \eqref{eq:24} and Lemma \ref{le:2}, we arrive at 
\begin{align} \label{eq:81}
\vep^2 \partial_\nu w_{\omega,\vep} = \frac{\rho_1\vep^2}{\rho_0} \left[\frac{{\vep^2\omega^2}{c_0^{-2}} \langle S_0^{-1}K^{(2)} 1,  S^{-1}_01\rangle_{S_0} u_{\omega}^{in}(y_0)}{\zeta + \frac{\rho_1}{\rho_0}-\frac{\omega^2 |\Omega|}{\mathcal C_\Omega c^2_0}-i\frac{\omega^3|\Omega|}{4\pi c^3_0}\vep} S^{-1}_01 + Res\right],
\end{align}
where $Res$ satisfies 
\begin{align} \label{eq:75}
\left\|Res\right\|_{{H^{-\frac12}(\Gamma)}} \le \frac{C_{d_{I,\max}}\left(\vep^{\frac 5 2}+ \vep^4\|w_{\omega,\vep}\|_{L^2(\Omega)}
\right)}{\left|\zeta + \frac{\rho_1}{\rho_0} -\frac{\omega^2|\Omega|}{\mathcal C_\Omega c^2_0} - i\frac{\omega^3|\Omega|}{4\pi c^3_0}\vep\right|}.
\end{align}

Secondly, we estimate $\|w_{\omega,\vep}\|_{L^2(\Omega)}$. It follows from \eqref{eq:18} that
\begin{align*}
&\left(\mathbb I-\left(\frac{1}{c^2_1}-\frac{1}{c_0^2}\right)\vep^2\omega^2 N_{\vep\omega/c_0}\right)w_{\omega,\vep}  = w_{\omega}^{in}-\left(\frac{\rho_0}{\rho_1\vep^2}-1\right) SL_{\vep\omega/c_0} \partial_\nu w_{\omega,\vep} \quad \textrm{in}\; \Omega.
\end{align*}
By \eqref{eq:116}, we readily obtain 
\begin{align*}
\left\|\left(\mathbb I-(c^{-2}_1-c^{-2}_0)\vep^2\omega^2N_{\vep\omega/c_0}\right)^{-1}\right\|_{L^2(\Omega), L^2(\Omega)} \le C_{d_{I,\max}}.
\end{align*}
From this, we use \eqref{eq:127} to get
\begin{align}\label{eq:82}
\|w_{\omega,\vep}\|_{L^2{(\Omega)}} \le C_{d_{I,\max}} \|w^{in}_{\omega}\|_{L^2{(\Omega)}} + C_{d_{I,\max}}\left|\frac{\rho_0}{\rho_1\vep^2}-1\right| \|\partial_\nu w_{\omega,\vep}\|_{H^{-\frac 12}(\Gamma)}.
\end{align}
Since ${\vep^{2}}/{\left|\zeta + \frac{\rho_1}{\rho_0} -\frac{\omega^2|\Omega|}{\mathcal C_\Omega c^2_0} - i\frac{\omega^3|\Omega|}{4\pi c^3_0}\vep\right|} \le C_{d_{I,\max}} \vep$, combining \eqref{eq:81}, \eqref{eq:75} and \eqref{eq:82} leads to
\begin{align} 
\|w_{\omega,\vep}\|_{L^2{(\Omega)}} \le C_{d_{I,\max}} |w^{in}_{\omega}\|_{H^1{(\Omega)}}& + C_{d_{I,\max}} \frac{|u^{ {in}}_{\omega}(y_0)|}{\left|\zeta
+ \frac{\rho_1}{\rho_0} - \frac{\omega^2 |\Omega|}{\mathcal C_\Omega c^2_0}-i\frac{\omega^3|\Omega|}{4\pi c^3_0}\vep\right|} \notag \\
&+ C_{d_{I,\max}}\frac{\vep^{\frac 12}}{\left|\zeta
+ \frac{\rho_1}{\rho_0}-\frac{\omega^2 |\Omega|}{\mathcal C_\Omega c^2_0}-i\frac{\omega^3|\Omega|}{4\pi c^3_0}\vep\right|} \label{eq:30}
\end{align}
for sufficiently small $\vep>0$. 

With the help of \eqref{eq:81}, \eqref{eq:75}, \eqref{eq:30}, statements \eqref{a3} and \eqref{a4} of Lemma \ref{le:7}, Lemma \ref{le:12} and Lemma \ref{le:6}, we  conclude from \eqref{eq:28} that when $c_1\ne c_0$, $u^{sc}_\omega(\vep)$ has the asymptotic expansion \eqref{eq:4} with the remainder term satisfying \eqref{eq:8} uniformly with respect to all $\omega \in I$. Hence, the proof is completed.

\end{proof}

\section{Resolvent's asymptotics of the scaled Hamiltonian}\label{section-4}

This section is devoted to proving Theorem \ref{th:2} and Corollary \ref{th:3}. It should be noted that the Lippmann-Schwinger equation and the spectral properties of $K^*_0$ are also crucial elements in deriving the uniform asymptotic expansion of the resolvent operator $R^H_{\rho_\vep,k_\vep}$.  

We begin by introducing the Lippmann-Schwinger equation corresponding to the resolvent $R^H_{\rho_\vep,k_\vep}(z)$. Let $\alpha>1/2$. For any $f\in L_{\alpha}^2(\R^3)$ and $z\in \overline{\CC_+}\backslash\{0\}$, denote by $v_z^f: = R^H_{\rho_0,k_0}(z)f$ and $u_{z,\vep}^f:=R^H_{\rho_\vep, k_\vep}(z)f$. It is known that $v_z^f \in H_{-\alpha}^{2}(\R^3)$ and $u_{z,\vep}^f \in H_{-\alpha}^{2}(\R^3 \backslash \Gamma_\vep) \cap H^1_{loc}(\R^3)$ solves
\begin{align}
& k_0\nabla \cdot \frac{1}{\rho_0} \nabla v_z^{f}  + z^2  v_z^f = -f \quad \textrm{in}\; \mathbb R^3 \label{eq:128} 
\end{align}
and
\begin{align*}
&k_\vep\nabla \cdot \frac 1 {\rho_\vep} \nabla u_{z,\vep}^f + z^2 u_{z,\vep}^f = -f \quad  \text{in}\;\R^3,
\end{align*}
respectively. Therefore, employing Green formulas leads to the following 
Lippmann-Schwinger equation
\begin{align}
& u^f_{z,\vep}(x)= v^f_z(x) + \left(\frac{1}{c^2_1} - \frac{1}{c^2_0} \right)z^2\int_{\Omega_\vep} \frac{e^{i{z}|x-y|/c_0}}{4\pi|x-y|}u^f_{z,\vep}(y)dy + \left(\frac{1}{c^2_1}-\frac{1}{c^2_0}\right)\int_{\Omega_\vep} \frac{e^{i{z}|x-y|/c_0}}{4\pi|x-y|} f(y) dy \notag\\
&\quad\quad \quad \;-\left(\frac{\rho_0}{\rho_1\vep^2}-1\right)\int_{\Gamma_\vep}\frac{e^{i{z}{|x-y|}/c_0}}{4\pi|x-y|} \partial_\nu u^f_{z,\vep}(y)d\sigma(y), \quad x\in \R^3 \backslash \Gamma_\vep,  \label{eq:13}
 \end{align}
where the value $u^f_{z,\vep}$ within $\Omega_\vep$ and the normal derivative $\partial_\nu u^f_{z, \vep}$ on $\Gamma_\vep$ are determined by
\begin{align} \label{eq:35}
\left(\mathbb I-\left(\frac{1}{c^2_1} - \frac{1}{c_0^2}\right)z^2N_{\Omega_\vep, z/c_0}\right)u^f_{z,\vep}&+\left(\frac{\rho_0}{\rho_1\vep^2}-1\right)\int_{\Gamma_\vep}\frac{e^{i{z}{|x-y|}/c_0}}{4\pi|x-y|} \partial_\nu u^f_{z,\vep}(y)d\sigma(y) \notag\\
&= v_z^{f} + \left(\frac{1}{c^2_1}-\frac{1}{c^2_0}\right) N_{\Omega_\vep, z/c_0} f, \quad \textrm{in}\; \Omega_\vep
\end{align}
and 
\begin{align}
&\left(\frac{1}{c^2_0}-\frac{1}{c^2_1}\right)z^2\partial_\nu N_{\Omega_\vep, z/c_0} u^f_{z,\vep} + \frac{\rho_0}{\rho_1\vep^2} \left(\frac 12\left(1 + \frac {\rho_1\vep^2} {\rho_0}\right) \mathbb I + \left(1-\frac {\rho_1\vep^2}{\rho_0}\right)K_{\Gamma_\vep, z/c_0}^*\right)\partial_\nu u^f_{z,\vep} \notag \\
&\qquad \qquad \qquad \; \qquad \qquad \qquad = \partial_\nu v_z^{f} + \left(\frac{1}{c^2_1}-\frac{1}{c^2_0}\right)\partial_\nu N_{\Omega_\vep, z/c_0} f \quad \textrm{on}\; \Gamma_\vep. \label{eq:12}
\end{align}

Consider the scaled functions $\widetilde v_{z,\vep}^{f}(y):= \left(v_z^{f}\circ \Phi_\vep\right)(y)$, $\widetilde u^{f}_{z,\vep}(y):=\left(u_{z,\vep}^{f}\circ \Phi_\vep\right)(y)$ and $\widetilde f (y) := \left(f\circ \Phi_\vep\right)(y)$ for $y\in \R^3$. The following lemma will investigate the properties of these scaled functions $\widetilde v_{z,\vep}^{f}$, $\widetilde u_{z,\vep}^{f}$ and $\widetilde f$.

\begin{lemma}\label{le:14}
Let  $z\in \overline{\CC_+}\backslash \{0\}$ and $\vep > 0$. The following arguments hold true.
\begin{enumerate}[(a)]
\item \label{g1}
For every $f\in L_{\alpha}^2(\R^3)$, we have
\begin{align*}
&\widetilde u^{f}_{z,\vep}= \widetilde v^{f}_{z,\vep} + \left(\frac{1}{c^2_1} - \frac{1}{c^2_0} \right)\vep^2z^2 N_{\vep z/c_0}\widetilde u^{f}_{z,\vep} + \vep^2\left(\frac{1}{c^2_1} - \frac{1}{c^2_0}\right)N_{\vep z/c_0} \widetilde f\notag\\
&\quad\quad \quad \;-\left(\frac{\rho_0}{\rho_1\vep^2}-1\right) SL_{\vep z/c_0}\partial_\nu \widetilde u^{f}_{z,\vep}(y), \quad \mathrm{in}\; \R^3 \backslash \Gamma, 
\end{align*}
where the value $\widetilde u^{f}_{z,\vep}$ within $\Omega$ and the normal derivative $\partial_\nu \widetilde u^{f}_{z, \vep}$ on $\Gamma$ are determined by
\begin{align}
&\left(\mathbb I-\left(\frac{1}{c^2_1}- \frac{1}{c_0^2}\right)\vep^2z^2N_{\vep z/c_0}\right) \widetilde u^{f}_{z,\vep} + \left(\frac{\rho_0}{\rho_1\vep^2}-1\right)SL_{\vep z/c_0} \partial_\nu \widetilde u^{f}_{z,\vep}\notag\\
& \qquad \qquad \qquad \qquad = \widetilde v_{z,\vep}^{f} + \vep^2\left(\frac{1}{c^2_1} - \frac{1}{c^2_0}\right) N_{\vep z/c_0} \widetilde f  \quad \textrm{in}\; \Omega \label{eq:87}
\end{align}
and 
\begin{align}
\left(\frac{1}{c^2_0}-\frac{1}{c^2_1}\right)\vep^2 z^2\partial_\nu N_{\vep z/c_0} \widetilde u^{f}_{z,\vep}+\frac{\rho_0}{\rho_1\vep^2} \Lambda^{(2)}_{\vep z,\vep} \partial_\nu \widetilde u^{f}_{z,\vep}  = \partial_\nu \widetilde v_{z,\vep}^{f} + \vep^2\left(\frac{1}{c^2_1} - \frac{1}{c^2_0}\right) \partial_\nu N_{\vep z/c_0} \widetilde f \;\; \textrm{on}\; \Gamma.  \label{eq:99}
\end{align}
Here, the operator $\Lambda^{(2)}_{\vep z,\vep}$ is defined by  \eqref{eq:23}.

\item \label{g3} For every $f\in L_{\alpha}^2(\R^3)$ and $\vep > 0$, we have 
\begin{align} \label{eq:124}
&\left\|\gamma\left(\widetilde v^{f}_{z,\vep}- v_{z}^{f}(y_0)\right)\right\|_{H^{\frac32}(\Gamma)} \le C\frac{1+|z|^{2}}{|z|}\vep^{\frac 12}\|f\|_{L^2_{\alpha}(\mathbb R^3)},\\
&\left|v_{z}^{f}(y_0)\right| \le C\frac{1+|z|^{2}}{|z|}\|f\|_{L^2_{\alpha}(\mathbb R^3)}. \label{eq:125}
\end{align}
Here, $C$ is a positive constant independent of $\vep$ and $z$.

\item \label{g4}
For every $f\in L_{\alpha}^2(\R^3)$ and $\vep > 0$, we have
\begin{align}\label{eq:90}
\langle \partial_\nu N_{\vep z/c_0} \widetilde u^{f}_{z,\vep}, S^{-1}_01 \rangle_{S_0} &= -\frac{\vep^2 z^2}{\mathcal C_\Omega c^2_0}\int_\Omega \int_\Omega \frac{e^{i{\vep z}|x-y|/c_0}}{4\pi|x-y|} \widetilde u^{f}_{z,\vep}(y)dxdy + \frac{c^2_1}{\vep^2 z^2}\left\langle \partial_\nu \widetilde u^{f}_{z,\vep}, S^{-1}_01 \right\rangle_{S_0} \notag \\
&+ \frac{1}{\mathcal C_\Omega z^2}\int_\Omega \widetilde f(y) dy
\end{align}
and
\begin{align} \label{eq:105}
\langle \partial_\nu N_{\vep z/c_0} \widetilde f, S^{-1}_01 \rangle_{S_0} = -\frac{\vep^2 z^2}{\mathcal C_\Omega c^2_0}\int_\Omega \int_\Omega \frac{e^{i{\vep z}|x-y|/c_0}}{4\pi|x-y|} \widetilde f(y)dxdy -  \frac{1}{\mathcal C_\Omega}\int_\Omega \widetilde f(y) dy.
\end{align}
\end{enumerate}
\end{lemma}

\begin{proof}
\eqref{g1} It is easy to verify that $\widetilde v_{z,\vep}^{f} \in  H_{-\alpha}^{2}(\R^3)$ and $\widetilde u_{z,\vep}^{f} \in  H_{-\alpha}^{2}(\R^3 \backslash \Gamma) \cap H^1_{loc}(\R^3)$ satisfy 
\begin{align}
& k_0\nabla \cdot \frac{1}{\rho_0} \nabla \widetilde v_{z,\vep}^{f}  + \vep^2 z^2  \widetilde v_{z,\vep}^{f}= -\vep^2 \widetilde f,\; \qquad \; \text{in}\; \mathbb R^3 \label{eq:94}
\end{align}
and 
\begin{align}
k_\vep\circ \Phi_\vep \nabla  \cdot \frac 1 {\rho_\vep \circ \Phi_\vep} \nabla \widetilde u_{z,\vep}^{f} + \vep^2z^2 \widetilde u_{z,\vep}^{f} = -\vep^2 \widetilde f,\;  \quad 
\;\;\; \;\; \; \text{in}\;\R^3,  \label{eq:31}
\end{align}
respectively. Therefore, the assertion of this statement easily follows from \eqref{eq:13}, \eqref{eq:35} and \eqref{eq:12}.

\eqref{g3} 
Since $v_{z}^{f}$ is the solution of equation \eqref{eq:128},  
using statement \eqref{a1} of Lemma \ref{le:7} and Lemma \ref{le:12} yields
\begin{align*}
\left\|\gamma\left(\widetilde v^{f}_{z,\vep} - v_{z}^{f}(y_0)\right)\right\|_{H^{\frac32}(\Gamma)} \le  C\vep^{\frac 12}\|v^{f}_{z}\|_{H^2(B_1(y_0))}\le C\frac{1+|z|^{2}}{|z|}\vep^{\frac 12}\|f\|_{L^2_{\alpha}(\mathbb R^3)}.
\end{align*}
This implies \eqref{eq:124}.
Moreover, it follows from inequality \eqref{eq:11} and Lemma \ref{le:12} that \eqref{eq:125} holds.

\eqref{g4} Proceeding as in the derivation of \eqref{eq:86}, we can apply \eqref{eq:31} to get \eqref{eq:90} and \eqref{eq:105}.
\end{proof}

In the sequel, we prepare several important estimates before proving Theorem \ref{th:2} and Corollary \ref{th:3}.

\begin{lemma}\label{le:13}
Let $\vep > 0$ and $z\in \overline{\CC_+}\backslash \{0\}$. 
Assume that $\alpha>1/2$. For every $f\in L^2_{\alpha}(\R^3)$, we have 
\begin{align}
&\left\|R^H_{\rho_0,k_0}(z)f - R^H_{\rho_0,k_0}(z)f_{a,\vep}\right\|_{L^2_{-\alpha}(\R^3)} \le C\vep^{\frac 32}\|R_{z/c_0}\|_{L^2_{\alpha}(\R^3), H^2_{-\alpha}(\R^3)}\|f\|_{L^2_{\alpha}(\R^3)}, \label{eq:14}\\
&\left|\left(R^H_{\rho_0,k_0}(z)f\right)(y_0) -\left(R^H_{\rho_0,k_0}(z)f_{a,\vep}\right)(y_0)\right|\le  C\vep^{\frac 12} \|f\|_{L^2_{\alpha}(\R^3)}. \label{eq:88}
\end{align}
Here, $f_{a,\vep}$ is defined by 
\begin{align} \label{eq:15}
f_{a,\vep}(x) = \begin{cases}
f(x) &  x \in \R^3 \backslash \Omega_\vep,\\
a\vep^2 f(x) & x\in \Omega_\vep,
\end{cases}
\quad \textrm{for}\; f\in L_{\alpha}^2(\R^3)
\end{align}
and $C$ is a constant independent of $\vep$ and $z$.
\end{lemma}

\begin{proof}
First, we prove that \eqref{eq:14} holds. A straightforward calculation gives 
\begin{align}\label{eq:25}
\left(R^H_{\rho_0,k_0}(f-f_{a,\vep})\right)(x) = -\int_{\Omega_\vep} \frac{e^{iz|x-y|/c_0}}{4\pi|x-y|}\frac{(1-a\vep^2)}{c_0^2}f(y)dy, \quad x\in \R^3.
\end{align}
Thus, for any $g \in L_{\alpha}^2(\R^3)$, we have 
\begin{align}
\int_{\R^3}\left(R^H_{\rho_0,k_0}(f-f_{a,\vep})\right)(x) g(x) dx  = -\int_{\Omega_\vep} \frac{1-a\vep^2}{c^2_0} f(y) (R_{z/c_0} g)(y)dy. \label{eq:33}
\end{align}
Combining \eqref{eq:11} and \eqref{eq:33} gives 
\begin{align*} 
\left|\int_{\R^3}\left(R^H_{\rho_0,k_0}(f-f_\vep)\right)(x) g(x) dx\right| & \le C \|R_{z/c_0} g\|_{L^{\infty}(\Omega_\vep)}\int_{\Omega_\vep} |f(y)|dy \\
& \le C\vep^{\frac32}\|R_{z/c_0}\|_{L^2_{\alpha}(\R^3), H^2_{-\alpha}(\R^3)} \|f\|_{L^2_{\alpha}(\R^3)}|g\|_{L^2_{\alpha}(\R^3)},
\end{align*}
whence \eqref{eq:14} follows.

Second, we focus on the estimation of \eqref{eq:88}. It follows from \eqref{eq:25} that 
\begin{align*}
\left|\left(R^H_{\rho_0,k_0}(f-f_{a,\vep})\right)(y_0)\right| = \frac{|1-a\vep^2|}{c_0^2}\left|\int_{\Omega_\vep} \frac{e^{iz|y_0-y|/c_0}}{4\pi|y_0-y|}f(y)dy\right|.
\end{align*}
By Cauchy--Schwartz inequality, we have
\begin{align*}
\left|\left(R^H_{\rho_0,k_0}(f-f_{a,\vep})\right)(y_0)\right| &\le C \|f\|_{L^2_{\alpha}(\R^3)} \frac 1{4\pi}\left(\int_{\Omega_\vep}\frac{1}{|y_0-y|^2}dy\right)^{\frac 12}\le C \vep^{\frac 12} \|f\|_{L^2_{\alpha}(\R^3)}.
\end{align*}
This directly implies that \eqref{eq:88} holds.

\end{proof}

\begin{lemma}\label{le:11}
Let $z\in \overline{\CC_+}$ with $|z|<1$. Suppose that $\vep > 0$ is sufficiently small such that $\vep z$ is not a Dirichlet eigenvalue of $-\Delta$ in $\Omega$ and that $\vep < 1/\sup_{x\in \Omega}|x-y_0|$. The following arguments hold true.
\begin{enumerate}[(a)]
\item \label{e1}
For every $f\in H_{\textrm{loc}}^2(\R^3)$, we have 
\begin{align}
& \left\|S_{\vep z}^{-1} \gamma N_{\vep z} \left(f\circ \Phi_\vep\right)- f(y_0)S_0^{-1} \gamma N_0 1\right\|_{H^{-\frac{1}2}(\Gamma)} \le C\vep^{\frac 12}\|f\|_{H^2(B_1(y_0))}, \label{eq:117}\\
&\left|\langle S_{\vep z}^{-1}\gamma N_{\vep z} \left(f\circ \Phi_\vep\right) , S^{-1}_01\rangle_{S_0} - \frac{f(y_0)|\Omega|}{\mathcal C_\Omega}\right| \le C\vep^{\frac 12}\|f\|_{H^2(B_1(y_0))},\quad \textrm{as}\; \vep \rightarrow 0.\label{eq:106}
\end{align}

\item \label{e2} Let $a \in \R_+ $. For every $f\in L^2(\R^3)$, we have
\begin{align} 
& \left\|S_{\vep z}^{-1} \gamma N_{\vep z} f_{a,\vep}\circ \Phi_\vep\right\|_{H^{-\frac12}(\Gamma)} \le C\vep^{\frac 12}\|f\|_{L^2(\R^3)}  \quad \textrm{as}\; \vep \rightarrow 0, \label{eq:131}
\end{align} 
where $f_{a,\vep}$ is defined in \eqref{eq:15}.
\end{enumerate}
Here, $C$ is a constant independent of $\vep$ and $z$.
\end{lemma}

\begin{proof}
\eqref{e1} 
Since $f\in H^2_{\textrm{loc}}(\R^3)$, it follows from \eqref{eq:11} that $f$ is continuous at $y_0$.  
Similarly as in the derivation \eqref{eq:26}, we have 
\begin{align*}
    \|f \circ \Phi_{\vep}- f(y_0)\|_{H^2(\Omega)} \le C\vep^{\frac 12}\|f\|_{H^2(B_1(y_0))}.
\end{align*}
Further, it is easy to verify
\begin{align*}
\langle S_0^{-1}\gamma N_0 1, S^{-1}_01\rangle_{S_0} = \frac{1}{\mathcal C_\Omega} \int_\Omega \int_\Gamma \left(S^{-1}_01\right)(x)\frac{1}{4\pi|x-y|}d\sigma(x)dy = \frac{|\Omega|
}{\mathcal C_\Omega}.
\end{align*}
Therefore, using statement \eqref{b4} of Lemma \ref{le:5} and \eqref{eq:78} gives that \eqref{eq:117} and \eqref{eq:106} hold.

\eqref{e2} By \eqref{eq:15}, we have
\begin{align} \label{eq:119}
(f_{a,\vep}\circ\Phi_\vep)(y) = f_{a,\vep}(y_0+\vep(y-y_0)) = a\vep^2 f(y_0+\vep(y-y_0)) = a\vep^2 (f \circ \Phi_\vep)(y), \quad y\in \Omega.
\end{align}
Since 
\begin{align}
\|f\circ \Phi_\vep\|_{L^2(\Omega)} \le \vep^{-\frac 3 2}\|f\|_{L^2(\R^3)}, \quad \textrm{for any}\; f\in L^2(\R^3), \label{eq:101}
\end{align}
inequality \eqref{eq:131} follows from \eqref{eq:116}, \eqref{eq:78} and \eqref{eq:119}. 
\end{proof}

We will utilize Lemmas \ref{le:14}, \ref{le:13} and \ref{le:11} to prove Theorem \ref{th:2} and Corollary \ref{th:3}.

\begin{proof}[Proof of Theorem \ref{th:2} and Corollary \ref{th:3}]
We note that $\chi_{1,\vep} h = h$ for every $h\in L^2_{\alpha,y_0}(\R^3)$ when $\vep$ is small enough, and that
\begin{align*}
\lim_{\vep\rightarrow 0}\frac{\varepsilon \omega^2 \mathcal C_\Omega}{\omega^2_M-\omega^2-i\varepsilon\frac{\omega^3 \mathcal C_\Omega}{4\pi c_0}}
= 
\begin{cases}
0,  &\mathrm{if}\; \omega \ne \pm \omega_M, \\
\pm i\frac{4\pi c_0}{\omega_M},  &\mathrm{if}\; \omega = \pm \omega_M.
\end{cases}
\end{align*}
Here, $\chi_{1,\vep}(x): = 1$ for $x\in \R^3 \backslash \Omega_\vep$ and $\chi_{1,\vep}(x):= \vep^2$ for $x\in \Omega_\vep$.
Thus, the results of Corollary \ref{th:3} are immediately derived from statement \eqref{1} of Theorem \ref{th:2}. Therefore, our focus will be primarily on proving Theorem \ref{th:2}. Throughout the proof, we assume that $\vep>0$ is sufficiently small.

For every $g\in L^2_{\alpha}(\R^3)$, we use statement \eqref{g1} of Lemma \ref{le:14} to get
\begin{align}
\int_{\R^3}\bigg(u_{z,\vep}^{f}(x)&-v_z^{f}(x)\bigg)g(x)dx
= \vep^3 \int_{\R^3}\left(\widetilde u_{z,\vep}^{f}(x)- \widetilde v_{z,\vep}^{f}(x)\right)g(y_0+\vep(x-y_0))dx \notag\\
&= \left(\frac{1}{c^2_1} - \frac{1}{c^2_0} \right)\vep^5z^2\int_{\Omega}\widetilde u_{z,\vep}^{f}(y)\left(R_{\vep z/c_0} \left(g\circ\Phi_\vep \right)\right)(y)dy \notag\\
& + \vep^5\left(\frac{1}{c^2_1} - \frac{1}{c^2_0}\right)\int_{\Omega}\widetilde f(y)\left(R_{\vep z/c_0}\left(g\circ \Phi_\vep\right)\right)(y)dy \notag\\
& - \vep^3\int_{\Gamma} \left(\frac{\rho_0}{\rho_1\vep^2}-1\right)\partial_\nu \widetilde u_{z,\vep}^f(y)\left(R_{\vep z/c_0}\left(g\circ \Phi_\vep \right)\right)(y) d\sigma(y), \quad f = h_{a,\vep}\;\textrm{or}\; h.\label{eq:96}
\end{align}
Here, $h_{a,\vep}$ is defined in \eqref{eq:15}. Furthermore, a straightforward calculation gives 
\begin{align*}
\left(R_{\vep z/c_0}\left(g\circ \Phi_\vep\right)\right)(y)&= \int_{\R^3} \frac{e^{i\vep z|x-y|/c_0}}{4\pi|x-y|}g(y_0+\vep(x-y_0))dx \\
&= \int_{\R^3} \vep\frac{e^{iz |y_0+ \vep (x-y_0)-(y_0+\vep(y-y_0))|/c_0}}{4\pi|y_0+ \vep (x-y_0)-(y_0+\vep(y-y_0))|} g(y_0+\vep(x-y_0))dx\\
& = \frac{1}{\vep^2} \int_{\R^3} \frac{e^{iz |t-(y_0+\vep(y-y_0))|/c_0}}{4\pi|t-(y_0+\vep(y-y_0))|} g(t)dt=\frac 1{\vep^{2}}\left(\left(\Phi_{\vep}\circ R_{z/c_0}\right) g\right)(y).
\end{align*}
From this, we can apply statement \eqref{a2} of Lemma \ref{le:7} and Lemma \ref{le:12} to get
\begin{align} \label{eq:100}
\|\vep^2 \left(R_{{\vep z}/{c_0}}\left(g\circ \Phi_\vep\right)\right)(y) - \left(R_{{z}/{c_0}}g\right)(y_0)\|_{H^2(\Omega)} \le 
C_{d_{V,\max}, d_{V,\min}} \vep^{\frac 12}\|g\|_{L_{\alpha}^2(\R^3)}.
\end{align}
The rest of the proof is divided into two parts: the first part involves proving statement \eqref{1} of Theorem \ref{th:2} and the second part addresses statement \eqref{2} of Theorem \ref{th:2}.

\textbf{Part 1:}
In this part, we first prove that for every $h \in L^2_{\alpha}(\R^3)$ and $g \in L^2_{\alpha}(\R^3)$,
\begin{align} \label{eq:36}
\int_{\R^3}(u_{z,\vep}^{h_{a,\vep}}(x)-v_z^{h_{a,\vep}}(x)) g(x)dx = \frac{\vep z^2 \mathcal C_\Omega}{\omega^2_M- z^2 -i\varepsilon\frac{z^3 \mathcal C_\Omega}{4\pi c_0}}v_z^{h_{a,\vep}}(y_0)\left(R_{z/c_0} g\right)(y_0) + \textrm{Rem}_{h_{a,\vep}}
\end{align}
with 
\begin{align} \label{eq:37}
|\textrm{Rem}_{h_{a,\vep}}| \le C_{d_{V,\max}, d_{V,\min}}\frac{\vep^{3/2}}{\left|\omega^2_M-z^2-i\varepsilon\frac{z^3\mathcal C_\Omega}{4\pi c_0}\right|}\|h_{a,\vep}\|_{L^2_{\alpha}(\mathbb \R^3)} \|g\|_{L_{\alpha}^2(\R^3)}
\end{align}
holding uniformly with respect to all $z\in V$. For this aim, we distinguish between two cases $c_1 = c_0$ and $c_1 \ne c_0$. 

\textbf{Case 1:} $c_1 = c_0$. In this case, setting $f = h_{a,\vep}$ in \eqref{eq:124}, \eqref{eq:125} and \eqref{eq:94}, and using \eqref{eq:51}, \eqref{eq:122}, \eqref{eq:112}, \eqref{eq:131} and Lemma \ref{le:8}, we can estimate
\begin{align}
&\left|\left\langle \partial_\nu \widetilde v_{z,\vep}^{h_{a,\vep}}, S^{-1}_01 \right\rangle_{S_0} -{\frac{\vep^2z^2}{c_0^{2}}\langle S_0^{-1}K^{(2)} 1, S^{-1}_01\rangle_{S_0} v_{z}^{h_{a,\vep}}(y_0)}\right| \le C_{d_{V,\max}} \vep^{\frac 52}\|h_{a,\vep}\|_{L^2_{\alpha}(\mathbb R^3)}. \label{eq:133}\\
& \left\|\partial_\nu \widetilde v_{z,\vep}^{h_{a,\vep}} - \mathcal P \partial_\nu \widetilde v_{z,\vep}^{h_{a,\vep}}\right\|_{H^{-\frac 12}(\Gamma)} \le C_{d_{V,\max}} \vep^{\frac 12}\|h_{a,\vep}\|_{L^2_{\alpha}(\mathbb R^3)}. \label{eq:134}
\end{align}
By employing \eqref{eq:99}, \eqref{eq:133}, \eqref{eq:134} and Lemma \ref{le:2}, we derive that
\begin{align}\label{eq:102}
\partial_\nu \widetilde u_{z,\vep}^{h_{a,\vep}} = \frac{\rho_1\vep^2}{\rho_0}\left[\frac{z^2{c_0^{-2}}\langle S_0^{-1}K^{(2)} 1,  S^{-1}_01\rangle_{S_0} v_{z}^{h_{a,\vep}}(y_0)}{\frac{\rho_1}{\rho_0}-\frac{z^2 |\Omega|}{\mathcal C_\Omega c^2_0}-i\frac{z^3|\Omega|}{4\pi c^3_0}\vep}  S^{-1}_01 
+ Res_{0}\right]\quad \text{on}\; \Gamma,
\end{align}
where $Res_0$ satisfies
\begin{align}\label{eq:103}
\left\|Res_{0}\right\|_{{H^{-\frac12}(\Gamma)}} \le \frac{C_{d_{V,\max}}\vep^{\frac 12}}{\left|\frac{\rho_1}{\rho_0} -\frac{z^2|\Omega|}{\mathcal C_\Omega c^2_0} - i\frac{z^3|\Omega|}{4\pi c^3_0}\vep\right|} \|h_{a,\vep}\|_{L^2_{\alpha}(\mathbb R^3)}.
\end{align}
Inserting \eqref{eq:102} and \eqref{eq:103} into \eqref{eq:96}, and using \eqref{eq:100} and Lemma \ref{le:6}, we obtain that \eqref{eq:36} and \eqref{eq:37} hold for the case when $c_1 = c_0$.

\textbf{Case 2: $c_1 \ne c_0$.} 
Subtracting $(1-c^2_1/c^2_0)\mathcal P\partial_\nu \widetilde u^{f}_{z,\vep}$ on both sides of \eqref{eq:99} and setting $f = h_{a,\vep}$, we have
\begin{align} \label{eq:38}
\frac{\rho_0}{\rho_1\vep^2}\left(\Lambda_{\vep z,\vep}^{(2)} + \vep^2 \zeta \mathcal P\right)\partial_\nu \widetilde u^{h_{a,\vep}}_{z,\vep} & = \partial_\nu \widetilde v^{h_{a,\vep}}_{z,\vep} + \left(\frac{1}{c^2_1}-\frac{1}{c^2_0}\right){\vep^2 z^2}\partial_\nu N_{\vep z/c_0} \widetilde u^{h_{a,\vep}}_{z,\vep} + \left(\frac{c^2_1}{c^2_0}-1\right)\mathcal P\partial_\nu \widetilde u^{h_{a,\vep}}_{z,\vep} \notag \\
& + \vep^2\left(\frac{1}{c^2_1} - \frac{1}{c^2_0}\right) \partial_\nu N_{\vep z/c_0} \widetilde h_{a,\vep} =: q_{z,\vep},
\end{align}
where $\zeta:=\rho_1\left(c^2_1/c^2_0-1\right)/\rho_0$. Setting $f= h_{a,\vep}$ in 
\eqref{eq:90} and \eqref{eq:105}, and applying \eqref{eq:116}, we get
\begin{align*}
&\left|\left\langle q_{z,\vep} - \partial_\nu \widetilde v_{z,\vep}^{h_{a,\vep}}, S^{-1}_01\right\rangle_{S_0}\right| \le C_{d_{V,\max}} \vep^4\|\widetilde u^{h_{a,\vep}}_{z,\vep}\|_{L^2(\Omega)},\\
&\left\|\left(\mathbb I - \mathcal P \right)\left(q_{z,\vep} - \partial_\nu \widetilde v_{z,\vep}^{h_{a,\vep}}\right)\right\|_{H^{-\frac12}(\Gamma)} \le C_{d_{V,\max}} \left(\vep^2 \|\widetilde h_{a,\vep}\|_{L^2(\Omega)} + \vep^4\|\widetilde u^{h_{a,\vep}}_{z,\vep}\|_{L^2(\Omega)}\right).
\end{align*}
From this, with the aid of \eqref{eq:119}, \eqref{eq:101}, \eqref{eq:133}, \eqref{eq:134} and \eqref{eq:38} and Lemma \ref{le:2}, we arrive at 
\begin{align}\label{eq:97}
\partial_\nu \widetilde u^{h_{a,\vep}}_{z,\vep} = \frac{\rho_1\vep^2}{\rho_0} \left[\frac{z^2c_0^{-2} \langle S_0^{-1}K^{(2)} 1, S^{-1}_01\rangle_{S_0} v_{z}^{h_{a,\vep}}(y_0)}{\zeta + \frac{\rho_1}{\rho_0}-\frac{z^2|\Omega|}{\mathcal C_\Omega c^2_0}-i\frac{z^3|\Omega|}{4\pi c^3_0}\vep} S^{-1}_01 + Res_1\right],
\end{align}
where $Res_1$ satisfies 
\begin{align} \label{eq:98}
\left\|Res_{1}\right\|_{{H^{-\frac12}(\Gamma)}} \le \frac{C_{d_{V_{\max}}}\left(\vep^{\frac 12}\left\|h_{a,\vep}\right\|_{L^2_{\alpha
}(\R^3)} + \vep^2 \|\widetilde u^{h_{a,\vep}}_{z,\vep}\|_{L^2(\Omega)}\right)}{\left|\zeta + \frac{\rho_1}{\rho_0} -\frac{z^2|\Omega|}{\mathcal C_\Omega c^2_0} - i\frac{z^3|\Omega|}{4\pi c^3_0}\vep\right|} .
\end{align}
Furthermore, by following the same procedure as the derivation of \eqref{eq:30}, we can use \eqref{eq:87}, \eqref{eq:97}, \eqref{eq:98} and statement \eqref{f2} of Lemma \ref{le:8} to get the estimate of $\widetilde u^{h_{a,\vep}}_{z, \vep}$ in $\Omega$, that is,
\begin{align*}
\|\widetilde u^{h_{a,\vep}}_{z,\vep}\|_{L^2{(\Omega)}} \le \frac{C_{d_{V,\max}}\left(\left|v_{z}^{h_{a,\vep}}(y_0)\right|+ \vep^{\frac 12}\left\|h_{a,\vep}\right\|_{L^2_{\alpha
}(\R^3)}\right)}{\left|\zeta + \frac{\rho_1}{\rho_0} -\frac{z^2|\Omega|}{\mathcal C_\Omega c^2_0} - i\frac{z^3|\Omega|}{4\pi c^3_0}\vep\right|}.
\end{align*}
Building upon the estimates of $\partial_\nu \widetilde u_{z,\vep}^{h_{a,\vep}}$ and $\widetilde u_{z,\vep}^{h_{a,\vep}}$ on $\Gamma$, we can utilize \eqref{eq:125}, \eqref{eq:101}, \eqref{eq:96}, \eqref{eq:100} and Lemma \ref{le:6} to obtain \eqref{eq:36} and \eqref{eq:37} for the case when $c_1 \ne c_0$.

Therefore, we obtain that equation \eqref{eq:36} holds with the remainder term $Rem$ satisfying \eqref{eq:37} uniformly with respect to all $z\in V$. This, together with the fact that $\|h_{a,\vep}\|_{L^2_{\alpha}(\mathbb \R^3)} \le \max(1,a)\|h\|_{L^2_{\alpha}(\mathbb \R^3)} $ directly implies
\begin{align}
\bigg\|\left(R^H_{\rho_\vep,k_\vep}(z)h_{a,\vep}\right)(x) &-\left(R^H_{\rho_0,k_0}(z) h_{a,\vep}\right)(x)\notag \\
&- \frac{\varepsilon z^2\mathcal C_\Omega}{\omega^2_M-z^2-i\varepsilon\frac{z^3 \mathcal C_\Omega}{4\pi c_0}}\left(R^H_{\rho_0,k_0}(z)h_{a,\vep}\right)(y_0)\frac{e^{iz|x-y_0|/c_0}}{4\pi|x-y_0|}\bigg\|_{L_{-\alpha}^2(\R^3)} \notag\\
& \le C_{d_{V,\max}, d_{V,\min}}\frac{\vep^{3/2}}{\left|\omega^2_M-z^2-i\varepsilon\frac{z^3\mathcal C_\Omega}{4\pi c_0}\right|}\|h\|_{L^2_{\alpha}(\mathbb \R^3)} \|g\|_{L_{\alpha}^2(\R^3)}. \label{eq:126}
\end{align}
Note that $\chi_{a,\vep} h = h_{a,\vep}$. From \eqref{eq:126} and Lemma \ref{le:13}, we conclude that the assertion of statement \eqref{1} of Theorem \ref{th:2} holds.

\textbf{Part 2:}
In this part, we assume that $h\in L^2_{\alpha}(\R^3)\cap H^2_{\textrm{loc}}(\R^3)$. Setting $f = h$ in \eqref{eq:124}, \eqref{eq:125} and \eqref{eq:94}, and utilizing \eqref{eq:51}, \eqref{eq:122}, \eqref{eq:11}, \eqref{eq:112}, Lemma \ref{le:8} and statement \eqref{e1} of Lemma \ref{le:11}, we have 
\begin{align}
&\left|\left\langle \partial_\nu \widetilde v_{z,\vep}^{h}, S^{-1}_01 \right\rangle_{S_0} -{\frac{\vep^2z^2}{c_0^{2}}\langle S_0^{-1}K^{(2)} 1, S^{-1}_01\rangle_{S_0} v_{z}^{h}(y_0)} + \vep^2 c_0^{-2}\mathcal C_{\Omega}^{-1}|\Omega|h(y_0)\right| \notag\\
&\qquad \qquad \qquad \qquad \qquad \qquad \qquad\;\; \le C_{d_{V,\max}}\vep^{\frac 52}\left(\|h\|_{L^2_{\alpha}(\mathbb \R^3)} + \|h\|_{H^2(B_1(y_0))}\right)\label{eq:138}
\end{align}
and 
\begin{align}
\left\|\partial_\nu \widetilde v_{z,\vep}^{h} - \mathcal P \partial_\nu \widetilde v_{z,\vep}^{h}\right\|_{H^{-\frac 12}(\Gamma)} \le C_{d_{V,\max}} \vep^{\frac 12}\left(\|h\|_{L^2_{\alpha}(\mathbb \R^3)} + \|h\|_{H^2(B_1(y_0))}\right). \label{eq:139}
\end{align}
To prove statement \eqref{2} of Theorem \ref{th:2}, it suffices to prove that for every $h\in L^2_{\alpha}(\R^3)\cap H^2_{\textrm{loc}}(\R^3)$ and $g\in L^2_{\alpha}(\R^3)$,
\begin{align}\label{eq:136}
\int_{\R^3}(u_{z,\vep}^{h}(x)-v_z^{h}(x)) g(x)dx & = \frac{\vep \mathcal C_\Omega}{\omega^2_M- z^2 -i\varepsilon\frac{z^3 \mathcal C_\Omega}{4\pi c_0}}\left( z^2v_z^{h}(y_0) + h(y_0)\right)\left(R_{z/c_0} g\right)(y_0)\quad \notag\\ 
& \qquad \qquad \qquad \qquad\qquad \qquad\qquad \qquad\qquad \qquad+ \text{Rem}_h
\end{align}
with
\begin{align}\label{eq:135}
|\textrm{Rem}_h| \le C_{d_{V,\max}, d_{V,\min}}\frac{\vep^{3/2}}{\left|\omega^2_M-z^2-i\varepsilon\frac{z^3\mathcal C_\Omega}{4\pi c_0}\right|}\left(\|h\|_{L^2_{\alpha}(\mathbb \R^3)} + \|h\|_{H^2(B_1(y_0))}\right)\|g\|_{L_{\alpha}^2(\R^3)}
\end{align}
holding uniformly with respect to all $z\in V$. 
In fact, using the same arguments as in the derivation of \eqref{eq:36} and \eqref{eq:37}  we can obtain that \eqref{eq:136} holds with the remainder term satisfying \eqref{eq:135} uniformly with respect to all $z \in V$. The key difference is that \eqref{eq:138} and \eqref{eq:139} 
serve as analogues of \eqref{eq:133} and \eqref{eq:134}, respectively.
\end{proof}

\section{Minnaert resonance as a pole of the scaled Hamiltonian}\label{app}
This section is devoted to establishing the relationship between the Minnaert frequency and the scattering resonances.

We begin by introducing an alternative definition of scattering resonances.
\begin{definition} \label{d1}
For each $\vep>0$ and $z \in \CC$, we denote
\begin{align*}
A_{\Omega_\vep,\vep}(z):=\begin{bmatrix}
&\mathbb I - \left(\frac{1}{c^2_1}-\frac{1}{c^2_0}\right)z^2~ N_{\Omega_\vep, z/c_0} & \left(\frac{\rho_0}{\rho_1\vep^2}-1\right) SL_{\Gamma_\vep, z/c_0}\\
&\left(\frac{1}{c^2_0}-\frac{1}{c^2_1}\right)z^2 \partial_\nu N_{\Omega_\vep, z/c_0} &\quad\frac{\rho_0}{\rho_1\vep^2} \left(\frac 12\left(1 + \frac {\rho_1\vep^2} {\rho_0}\right) \mathbb I + \left(1-\frac {\rho_1\vep^2}{\rho_0}\right)K_{\Gamma_\vep, z/c_0}^*\right)
\end{bmatrix}
\end{align*}
as a linear bounded operator from $L^2(\Omega_\vep)\times L^{2}(\Gamma_\vep)$ into itself.
We call $z$ a scattering resonance of the Hamiltonian $H_{\rho_\vep,k_\vep}$ for each fixed $\vep$ if the operator $A_{\Omega_\vep,\vep}(z)$ is not injective.
\end{definition}

\begin{remark} \label{re:2}
From Definition \ref{d1}, it can be seen that $z \in \mathbb C \backslash \{0\}$ is a scattering resonance if and only if there exists $(\phi_{z}, \psi_{z}) \in L^2(\Omega) \times L^2(\Gamma)$ such that 
\begin{align}
\mathcal A_{\Omega_\vep,\vep}(z) 
\begin{bmatrix}
\phi_{z}\\
\psi_{z} 
\end{bmatrix} = \begin{bmatrix}
0\\
0
\end{bmatrix}. \notag
\end{align}
Setting 
\begin{align}\label{eq:50}
u_\vep:= \left(\frac 1{c_1^{2}} - \frac{1}{c_0^{2}}\right)z^2 N_{\Omega_\vep, z/c_0} \phi_z - \left(\frac{\rho_0}{\rho_1\vep^2} - 1 \right) SL_{\Gamma_\vep, z/c_0} \psi_z, \quad \mathrm{in}\; \R^3 \backslash \Gamma_\vep.
\end{align}
Clearly,
\begin{align} \label{eq:48}
H_{\rho_\vep,k_\vep} u_\vep - z^2 u_\vep = 0\quad \mathrm{in}\; H_{\textrm{loc}}^2(\R^3 \backslash \Gamma_\vep) \cap H_{\textrm{loc}}^1(\R^3),
\end{align}
i.e., that 
\begin{align*}
&\Delta u_{\vep} + z^2c_0^{-2} u_{\vep} = 0 \quad {\rm{in}}\; \R^3 \backslash \Omega_\vep,\\
&\Delta u_{\vep} + z^2 c_1^{-2} u_{\vep} = 0 \quad {\rm{in}}\; \Omega_\vep,\\
& u^+_{\vep} = u^-_\vep, \quad \partial^+_{\nu} u_{\vep} = \frac{\rho_0}{\rho_1\vep^2}\partial^-_{\nu} u_{\vep} \quad {\rm{on}}\; \Gamma_\vep.
\end{align*}
Furthermore, it is easy to verify that $(1-\chi)SL_{\Gamma_\vep, z/c_0} \psi_z = R_{z/c_0} [\Delta, \chi] SL_{\Gamma_\vep, z/c_0} \psi_z$, where $\chi \in C_c^{\infty}(\R^3)$ satisfies $\chi = 1 $ in $B_r$. Here, $r>0$, independent of $\vep$, is chosen such that $\overline{\Omega_\vep} \subset B_r$. Therefore, we readily obtain that $u_\vep$, as specified in \eqref{eq:50}, is $z/c_0-$ outgoing. Here, we say that $u$ is $z-$ outgoing if there exist $g\in L_{\textrm{comp}}^2(\R^3)$ and $r>0$ such that $u = R_z g$ outside $B_r$, where $R_z$ is given by \eqref{eq:69}. We refer to \cite[Definition 3.32]{DM} and \cite{J19} for such a definition. For $z \in \overline{C_+}$, $z-$ outgoing field is the one that satisfies the classical outgoing radiation condition. For $z\in \mathbb C_-$, ``outgoing" refers to the analytic continuation of the classical outgoing solution from $\overline{\mathbb C_+}$ into $\mathbb C_-$ via $R_z$.
\end{remark}

Interestingly, Definition \ref{d1} is equivalent to Definition \ref{d2}. Before proving this equivalence, we state a lemma, which is a special case of Theorem 4.9 in \cite{DM}.

\begin{lemma} \label{le:4} 
Given $\vep>0$, $\lambda_\vep \in \mathbb C \backslash \{0\}$ is a scattering resonance of the Hamiltonian $H_{\rho_\vep,k_\vep}$ if and only if there exists $u_{\lambda_\vep} \in \mathcal D_{\mathrm{loc}}$ satisfying 
\begin{align*} 
&\Delta u_{\lambda_\vep} + \lambda_\vep^2c_0^{-2} u_{\lambda_\vep} = 0 \quad {\rm{in}}\; \R^3 \backslash \Omega_\vep, \\
&\Delta u_{\lambda_\vep} + \lambda_\vep^2 c_1^{-2} u_{\lambda_\vep} = 0 \quad {\rm{in}}\; \Omega_\vep,\\
& u^+_{\lambda_\vep} =  u^-_{\lambda_\vep}, \quad \partial^+_{\nu} u_{\lambda_\vep} = \frac{\rho_0}{\rho_1\vep^2}\partial^-_{\nu} u_{\lambda_\vep} \quad {\rm{on}}\; \Gamma_\vep, \\
& u_{\lambda_\vep}\; {\rm{is}}\; \lambda_\vep/c_0-{\rm{outgoing}}.
\end{align*}
\end{lemma}
Lemma 5.1 implies that if $\lambda_\vep$ is a resonance of the Hamiltonian $H_{\rho_\vep, k_\vep}$, then $-\overline{\lambda_\vep}$ is also.

\subsection{Equivalence of the two definitions}

Now we provide a proof to the equivalence of Definition \ref{d2} and  Definition \ref{d1}.  

\begin{proof}[Equivalence of Definition \ref{d1} and Definition \ref{d2}]
When $z=0$, due to the fact that $\big(1/2(1 + {\rho_1\vep^2}/{\rho_0})$ $\mathbb I + (1-{\rho_1\vep^2}/{\rho_0})K_{\Gamma_\vep, 0}^*\big)$ is invertible in $\mathcal L (L^2(\Gamma_\vep))$ for each fixed $\vep>0$, we easily find that $0$ is not a scattering resonance in the sense of Definition \ref{d1}. Furthermore, proceeding similarly to \cite[Theorem 4.19]{DM}, we readily obtain that $0$ is not a scattering resonance in the sense of Definition \ref{d2}. In what follows, we focus on the case of the nonzero resonances.

We begin by proving that any scattering resonance $z$ as defined in Definition \ref{d1} is a pole of the meromorphic extension of the $R^H_{\rho_\vep,k_\vep}(z)$. When $z$ is a point where $A_{\Omega_\vep,\vep}(z)$ fails to be injective, it follows from Remark \ref{re:2} that there exists $u_z$ that satisfies \eqref{eq:48} and is $z/c_0-$ outgoing. This, together with Lemma \ref{le:4} yields that $z$ is a pole of the meromorphic extension of the $R^H_{\rho_\vep,k_\vep}(z)$.

Conversely, since $H_{\rho_\vep,k_\vep}$ represents a type of black box Hamiltonian, when $z$ is identified as a pole of the meromorphic extension of the $R^H_{\rho_\vep,k_\vep}(z)$, Lemma \ref{le:4} implies that each resonance state $v_z$ corresponding to $z$ satisfies $H_{\rho_\vep,k_\vep} v_z - z^2 v_z = 0$ in $H_{\textrm{loc}}^2(\R^3 \backslash \Gamma_\vep) \cap H_{\textrm{loc}}^1(\R^3)$ and there exist $g\in L_{\textrm{comp}}^2(\R^3)$ and $R>0$ such that $v_z = R_z g$ outside $B_R$. Next, we prove that $(v_z|_{\Omega_\vep}, \partial_\nu v_z |_{\Gamma_\vep})$ solves 
\begin{align} \label{eq:123}
A_{\Omega_\vep,\vep}(z) 
\begin{bmatrix}
u_z|_{\Omega_\vep}\\
\partial_\nu u_z|{\Gamma_\vep}
\end{bmatrix}  = 0.
\end{align}
Observe that 
\begin{align}\label{eq:132}
\int_{\partial B_R} \frac{e^{i z|p-x|}}{|p-x|} \frac{\partial}{\partial{\nu (x)}}\int_{\textrm{supp}(g)}\frac{e^{iz|x-y|}}{|x-y|}g(y)dy - \int_{\textrm{supp}(g)}\frac{e^{iz|x-y|}}{|x-y|}g(y) dy\frac{\partial }{\partial{\nu (x)}}\frac{e^{i z|p-x|}}{|p-x|} d\sigma(x)= 0
\end{align}
for any $z \in \overline{\CC_+} \backslash \{0\}$ and $p \in \R^3 \backslash \overline{B_R}$. Here, $\textrm{supp}(g)$ denotes the compact support of $g$.
By analyticity of the functions in \eqref{eq:132} with respect to $z$, it can be deduced that \eqref{eq:132} holds for all $z \in \CC$ and $p \in \R^3 \backslash \overline{B_R}$. This, together with Green formulas directly yields that $(v_z|_{\Omega_\vep}, \partial_\nu v_z |_{\Gamma_\vep})$ solves \eqref{eq:123}.
Consequently, $z$ is point where $A_{\Omega_\vep,\vep}(z)$ fails to be injective.
\end{proof}

\subsection{Asymptotics of Minnaert resonances}

Now we are ready to investigate the asymptotic properties of the resonances. 

\begin{lemma} \label{le:a1}
Let $\vep>0$. The following properties hold true.
\begin{enumerate}[(a)] 

\item \label{z1} Suppose that $\vep > 0$ is sufficiently small. There exists a continuous curve $\vep\rightarrow z(\vep)\in \CC$ such that $Q_{\vep}(z(\vep))$ is not injective and $\lim_{\vep\rightarrow 0}z(\vep) = 0$. Here, $Q_{\vep}(z)$ is defined by
\begin{align*}
Q_{\vep}(z):=\begin{bmatrix}
&\mathbb I - \left(\frac{1}{c^2_1}-\frac{1}{c^2_0}\right)z^2~ N_{z/c_0} & \quad \left(\frac{\rho_0}{\rho_1\vep^2}-1\right) SL_{z/c_0}\\
\\
&\left(\frac{1}{c^2_0}-\frac{1}{c^2_1}\right)z^2 \partial_\nu N_{z/c_0} & \frac{\rho_0}{\rho_1\vep^2}\left(\frac 12\left(1 + \frac{\rho_1 \vep^2}{\rho_0}\right) \mathbb I + \left(1-\frac{\rho_1 \vep^2}{\rho_0}\right)K_{z/c_0}^*\right)
\end{bmatrix}.
\end{align*}
\item \label{z2} For any compact set $V\subset \CC$ containing $\pm \omega_M$, there exists $\eta >0$ such that when $\vep\in (0,\eta)$, $R^H_{\rho_\vep,k_\vep}(z)$ exhibits two unique scattering resonances $z_{\pm}(\vep)$ in $V$, called the Minnaert resonances, that satisfy
\begin{align} \label{eq:27} 
&z_{\pm}(\vep) = \pm \omega_M -i\frac{\omega^2_M \mathcal C_\Omega }{8\pi c_0}\vep + z_{\pm,res}(\vep),
\end{align}
where 
\begin{align}\label{eq:32}
|z_{\pm,res}(\vep)| \le C\vep^2, \quad \textrm{as} \; \vep \rightarrow 0.
\end{align}
Here, $\omega_M$ is a Minnaert frequency as defined in \eqref{eq:45}, $\mathcal C_\Omega$ is capacitance of $\Omega$ as defined in \eqref{eq:74}, and $C$ is a positive constant independent of $\vep$.
\end{enumerate}
\end{lemma}

\begin{proof}
First, we prove statement \eqref{z1}. We observe that
\begin{align*}
Q_{\vep}(z)\begin{pmatrix}
\phi\\
\psi
\end{pmatrix} = \widetilde Q_{\vep}(z) \begin{pmatrix}
\phi\\
\vep^{-2}\psi
\end{pmatrix}, \quad \textrm{for}\; (\phi,\psi) \in \mathcal L \left(L^2(\Omega)\times L^{2}(\Gamma)\right).
\end{align*}
Here, $\widetilde Q_{\vep}(z)$ is defined by
\begin{align*}
\widetilde Q_{\vep}(z):=\begin{bmatrix}
&\mathbb I - \left(\frac{1}{c^2_1}-\frac{1}{c^2_0}\right)z^2~ N_{z/c_0}   & \left(\frac{\rho_0}{\rho_1}-\vep^2\right) SL_{z/c_0}\\
& \left(\frac{1}{c^2_0}-\frac{1}{c^2_1}\right)z^2 \partial_\nu N_{z/c_0} & \quad \frac{\rho_0}{\rho_1}\left(\frac 12\left(1 + \frac{\rho_1 \vep^2}{\rho_0}\right) \mathbb I + \left(1-\frac{\rho_1 \vep^2}{\rho_0}\right)K_{z/c_0}^*\right)
\end{bmatrix}.
\end{align*}
Thus, $Q_{\vep}(z)$ and $\widetilde Q_{\vep}(z)$ share points where they are not injective. Clearly, $\widetilde Q_{\vep}(z)$ can be rewritten as
\begin{align*}
\widetilde Q_{\vep}(z) & = \begin{bmatrix}
&\mathbb I - \left(\frac{1}{c^2_1}-\frac{1}{c^2_0}\right)z^2~ N_{z/c_0} & \frac{\rho_0}{\rho_1}SL_{z/c_0}\\
&\left(\frac{1}{c^2_0}-\frac{1}{c^2_1}\right)z^2 \partial_\nu N_{z/c_0} & \qquad \frac{\rho_0}{\rho_1}\left(\frac{1}{2} + K_{z/c_0}^*\right)
\end{bmatrix} + \vep^2 \begin{bmatrix}
&0 & -SL_{z/c_0}\\
&0 & \qquad\frac{1}{2} - K_{z/c_0}^*
\end{bmatrix}\\
&=: E(z) + \vep^2 H(z).
\end{align*}
Our next aim is to demonstrate that $\widetilde Q_{\vep}(z)$ exhibits similar injective properties to those of $E(z)$ by using Gohberg-Sigal's theory (see, e.g., \cite[Theorem 1.15]{AK09}). We note that by statement \eqref{b4} of Lemma \ref{le:5} there exists  $\eta_0 \in (0,1)$ such that 
\begin{align} \label{eq:34} 
\left\|\left(\mathbb I - \left(\frac{1}{c^2_1}-\frac{1}{c^2_0}\right)z^2~ N_{z/c_0}\right)^{-1}-\mathbb I\right\|_{\mathcal L(L^2(\Omega))} \le C|z|^2, \quad z \in B_{\CC,\eta_0},
\end{align}
where $B_{\CC,s}:=\{z\in \CC: |z|<s\}$ for any $s\in \R_+$. Throughout the proof, $C$ is a positive constant independent of $z$ and $\vep$.
Therefore, for investigating the invertibity of $E(z)$, it suffices to prove that the Schur complement of $\mathbb I - \left({c^{-2}_1}- c^{-2}_0\right)z^2~ N_{z/c_0}$, defined by 
\begin{align*}
\mathbb M(z):= \frac{\rho_0}{\rho_1}\left(\frac{1}{2} + K_{z/c_0}^* - \left(\frac{1}{c^2_0}-\frac{1}{c^2_1}\right)z^2 \partial_\nu N_{z/c_0}\left(\mathbb I - \left(\frac{1}{c^2_1}-\frac{1}{c^2_0}\right)z^2~ N_{z/c_0}\right)^{-1} SL_{z/c_0}\right)
\end{align*}
is invertible in $\mathcal L(L^2(\Gamma))$. Recall that for every $\psi \in H^{-1/2}(\Gamma)$ can be represented by $\psi = \mathcal P\psi + (\mathbb I -\mathcal P)\psi =: \mathcal P \psi + \mathcal Q\psi$, where the operator $\mathcal P$ is defined in \eqref{eq:71}. Clearly, operators $\mathcal P$ and $\mathcal Q$ belong to $\mathcal L(L^2(\Gamma))$. Using the similar arguments as employed in the derivation of \eqref{eq:86}, we have 
\begin{align*}
\langle \partial_\nu N_{z/c_0} SL_{z/c_0}\phi, S^{-1}_01 \rangle_{S_0} = -\frac{z^2}{\mathcal C_\Omega c^2_0}\int_\Omega \left(N_{z/c_0} SL_{z/c_0} \phi\right) (y)dy + \frac{c^2_0}{z^2}\left\langle \frac{1}{2}\phi+ K_{z/c_0}^* \phi, S^{-1}_01 \right\rangle_{S_0},
\end{align*}
for any $\phi \in L^2(\Gamma)$. From this, \eqref{eq:34}, the identity $\mathcal P + \mathcal Q = \mathbb I$, and statement \eqref{f2} of Lemma \ref{le:8}, we can rewrite $\mathbb M(z)$ as
\begin{align}
\mathbb M(z) = \frac{\rho_0}{\rho_1}\bigg[\left(\left({c_0^2}{c_1^{-2}}-1\right) \mathcal P + \mathbb I\right)\left(\frac{1}{2} + K_{z/c_0}^*\right)
+ \left(\frac{1}{c^2_1}-\frac{1}{c^2_0}\right)\left(z^2 \mathcal Q \partial_\nu N_{z/c_0}SL_{z/c_0}\right) + \mathbb M_{Res}(z)\bigg], \label{eq:147}
\end{align}
where 
\begin{align*}
\|M_{Res}(z)\|_{\mathcal L(L^2(\Gamma))}\le C|z|^4, \quad |z|\in B_{\mathbb C,\eta_0}.
\end{align*}
Using statement \eqref{b2} of Lemma \ref{le:5} and the identities $\left(1/2 \mathbb I +K^*_0\right) \mathcal P = 0$, $\mathcal P\left(1/2 \mathbb I +K^*_0\right) \mathcal Q = 0$ and $\mathbb I = \mathcal P + \mathcal Q$, we have that for each $z \in B_{\CC,\eta_0}$,
\begin{align*}
\bigg\|\frac 12 \mathbb I+ K^{*}_{z/c_0} - \mathcal P \frac{z^2}{ c^2_0}K^{*,(2)}\left(\mathcal P + \mathcal Q\right) - \mathcal Q \bigg[\left(\frac{1}{2} + K^{*}_0\right) \mathcal Q + \frac{z^2}{ c^2_0}K^{*,(2)}\bigg]\bigg\|_{\mathcal L (L^2(\Gamma))} \le C|z|^3.
\end{align*}
This, together with \eqref{eq:91}, \eqref{eq:147} and the fact that $\mathcal Q(1/2 \mathbb I+K_0)\mathcal Q$ is invertible in $\mathcal L(\mathcal Q(L^2(\Gamma)))$ yields that there exists $\eta_1 \in (0,1)$ such that 
\begin{align*}
\left\|\mathbb M(z)^{-1}\right\|_{\mathcal L(L^2(\Gamma))} \le \frac{C}{|z|^2}, \quad z\in B_{\CC,\eta_1} \backslash \{0\}.
\end{align*} 
Based on the above discussions, it can be deduced that there exists $\eta_2 \in (0,1)$ and $\vep_{\eta_2}>0$ depending on $\eta_2$ such that
\begin{align*}
E^{-1}(z)& \in \mathcal L \left(L^2(\Omega)\times L^{2}(\Gamma)\right)\quad \textrm{in}\;B_{\CC,\eta_2}(0)\backslash \{0\},\\
& \mathrm{and}\; \vep^2\left\|E^{-1}(z)H(z)\right\|_{\mathcal L(L^2(\Omega)\times L^2(\Gamma))} < 1, \quad \textrm{on}\; \partial B_{\CC,\eta_2} \; \textrm{for}\; \vep \in (0,\vep_{\eta_2}).
\end{align*}
Furthermore, utilizing statement \eqref{b2} of Lemma \ref{le:5} and statement \eqref{f2} of Lemma \ref{le:8} again, we easily find
\begin{align*}
E(z) \begin{bmatrix}
    -\frac{\rho_0}{\rho_1} SL_{z/c_0} S_0^{-1} 1\\
    S_0^{-1} 1 
\end{bmatrix} = z^2h(z),
\end{align*}
where $h(z)$ is analytic in $L^2(\Omega)\times L^{2}(\Gamma)$ and $h(0) \ne (0,0)$. This implies that $0$ is a point where $E(z)$ fails to be injective and that the null multiplicity of $E(0)$ equals two (see section 1.1.3 in \cite{AK09} for the definition of null multiplicity of operators). Moreover, it immediately follows from Lemma \ref{le:5} that $E(z)$ and $H(z)$ are analytic families of operators for $z \in \CC$. Therefore, setting $A(z)= E(z),\; B(z)= \vep^2 H(z), V = B_{\mathbb C,\eta_2}$ in Theorem 1.15 in \cite{AK09}, we readily obtain that for each $\vep \in (0, \vep_{\eta_2})$, there exists $z(\vep)\in \CC$ such that $Q_{\vep}(z(\vep))$ is not injective and that its multiplicity in $B_{\CC,\eta_2}$, which is denoted by $\mathcal M (Q_{\vep}( z);\partial B(\CC,\eta_2))$ (see (1.9) in \cite{AK09} for the definition of the multiplicity of operators), satisfies
\begin{align}\label{eq:118}
\mathcal M (Q_{\vep}(z); \partial B(\CC,\eta_2)) = 2, \quad  \vep \in (0, \vep_{\eta_2}), \; z\in B_{\mathbb C,\eta_2}.
\end{align}
Similarly, we can obtain that $z(\vep)$ depends continuously on $\vep$ and $\lim_{\vep\rightarrow 0}z(\vep) = 0$. 

Second, we prove statement \eqref{z2}. Denote the scaled operator of $A_{\vep}(z)$ by 
\begin{align*}
A_{\vep}(z):=\begin{bmatrix}
&\mathbb I - \left(\frac{1}{c^2_1}-\frac{1}{c^2_0}\right)\vep^2 z^2~ N_{\vep z/c_0} & \quad \left(\frac{\rho_0}{\rho_1\vep^2}-1\right) SL_{\vep z/c_0}\\
\\
&\left(\frac{1}{c^2_0}-\frac{1}{c^2_1}\right)\vep^2 z^2 \partial_\nu N_{\vep z/c_0} & \frac{\rho_0}{\rho_1\vep^2}\left(\frac 12\left(1 + \frac{\rho_1 \vep^2}{\rho_0}\right) \mathbb I + \left(1-\frac{\rho_1 \vep^2}{\rho_0}\right)K_{\vep z/c_0}^*\right)
\end{bmatrix}.
\end{align*}
Clearly, $A_{\vep}(z)\in \mathcal L \left(L^2(\Omega)\times L^{2}(\Gamma)\right)$. Observe that for each $\vep >0$ $A_{\Omega_\vep,\vep}(z)$ and $A_{\vep}(z)$ share points where they are not injective. Moreover,  
\begin{align} \label{eq:61}
A_{\vep}(z) = Q_{\vep}(\vep z),  \quad \textrm{for}\; \vep>0, \; z\in\CC.
\end{align}
Therefore, to investigate the properties of the scattering resonance as defined in Definition \ref{d1}, it suffices to examine the properties of the operator $Q_{\vep}(z)$. 
We only focus on the proof of case of $c_1 = c_0$, since the case when $c_1 \ne c_0$ can be proved in a similar manner.
In the remainder of the proof, we assume that $\vep > 0$ is sufficiently small such that $\vep < \vep_{\eta_2}/(\max_{z\in V}|z|)$.

Given an element $z(\vep)$ from a bounded subset of $\CC$ such that $Q_{\vep}(z(\vep))$ is not injective, we know that there exists $(\phi_\vep, \psi_\vep) \in L^2(\Omega)\times L^{2}(\Gamma)$ such that 
\begin{align*}
Q_{\vep}(z(\vep))
\begin{pmatrix}
\phi_\vep\\
\psi_\vep
\end{pmatrix} = 0.
\end{align*}
This also implies $\left(Q_{\vep}(z(\vep))(\phi_\vep, \psi_\vep)^{T}\right)\cdot (0, 1)=0$, i.e.
\begin{equation}\label{original-characteristic-equation}
\left[\frac{\rho_1 \vep^2}{\rho_0}\mathbb I  + \left(1-\frac{\rho_1 \vep^2}{\rho_0}\right)\left(\frac 12 \mathbb I+ K_{z(\vep)/c_0}^*\right)\right]\psi_\vep = 0 \quad \textrm{on}\; \Gamma.
\end{equation}
With the decomposition $\psi_\vep = \mathcal P \psi_\vep + \mathcal Q \psi_\vep$, it follows from \eqref{original-characteristic-equation} that
\begin{align} \label{eq:137}
\mathcal B(z(\vep)) \psi_\vep = \mathcal P \psi_\vep, 
\end{align}
where $\mathcal B(z(\vep))$ is defined by
\begin{align} \label{eq:142}
\mathcal B(z(\vep)) := \frac{\rho_1 \vep^2}{\rho_0}\mathbb I  + \mathcal P + \left(1-\frac{\rho_1\vep^2}{\rho_0}\right)\left(\frac 12 \mathbb I+ K_{ z(\vep)/c_0}^*\right).
\end{align}
Since $\lim_{\vep\rightarrow 0}z(\vep) = 0$, by utilizing statement \eqref{b2} of Lemma \ref{le:5}, we find
\begin{align*}
\left\|\mathcal B(z(\vep)) - \mathcal P -\frac {1}2 - K^*_0\right\|_{\mathcal L \left(L^{2}(\Gamma)\right)} \le C|z(\vep)|^2.
\end{align*}
With the aid of the fact that $\mathcal P + {1}/2 + K^*_0$  has an inverse in $\mathcal L(L^{2}(\Gamma))$, we have 
\begin{align} \label{eq:145}
\left\|\left(\mathcal B(z(\vep))\right)^{-1} \right\|_{\mathcal L \left(L^{2}(\Gamma)\right)} \le C.
\end{align}
Thus, we deduce from \eqref{eq:137} that
\begin{align*}
\mathcal Q \psi_\vep = \left(\mathcal B(z(\vep))\right)^{-1} \mathcal P\psi_\vep - \mathcal P\psi_\vep.
\end{align*}
This, together with \eqref{eq:140} gives 
\begin{align} \label{eq:141}
\langle\left(\mathcal B(z(\vep))\right)^{-1} \mathcal P\psi_\vep, S^{-1}_01\rangle_{S_0} =\langle \mathcal P \psi_\vep, S^{-1}_01\rangle_{S_0}.
\end{align}
Setting 
\begin{align}\label{def-l}
l_\vep:= \left(\mathcal B(z(\vep))\right)^{-1} S^{-1}_01,
\end{align}
then we rewrite \eqref{eq:141} as
\begin{align}\label{l-psi-0}
\langle l_\vep, S^{-1}_01\rangle_{S_0}
= 1.
\end{align}
Combining \eqref{eq:142}, \eqref{def-l}, \eqref{l-psi-0} and statement \eqref{b2} of Lemma \ref{le:5} gives
\begin{align}
0& = \left[\frac{\rho_1 \vep^2}{\rho_0}\mathbb I +  \left(1-\frac{\rho_1 \vep^2}{\rho_0}\right)\left(\frac 12 \mathbb I+ K_{z(\vep)/c_0}^*\right)
\right]l_\vep \notag\\
&= \left[\frac{\rho_1 \vep^2}{\rho_0} \mathbb I + \left(1-\frac{\rho_1 \vep^2 }{\rho_0}\right)\left(\frac 12 \mathbb I+ K^*_0 + \frac{(z(\vep))^2}{ c^2_0}K^{*,(2)} + \frac{(z(\vep))^3}{c^3_0} K^{*,(3)} + \mathcal R^{(1)}_{\textrm{res},\vep}\right) \right]l_\vep,\label{eq:143}
\end{align}
where 
\begin{align}\label{eq:144}
\left\|\mathcal R^{(1)}_{\textrm{res},\vep}\right\|_{\mathcal L\left(L^2(\Gamma)\right)} \le C |z(\vep)|^4.
\end{align}
With the aid of \eqref{eq:143}, \eqref{eq:144} and the identities $\left(1/2 \mathbb I +K^*_0\right) \mathcal P = 0$ and $\mathbb I = \mathcal P + \mathcal Q$, we have 
\begin{align*}
&\left(\frac 12 \mathbb I+ K_{0}^*\right) \mathcal Q l_\vep  \\
&= -\left[\frac{\rho_1 \vep^2}{\rho_0}\mathbb I-\frac{\rho_1 \vep^2}{\rho_0}\left(\frac 12 \mathbb I+ K_{0}^*\right) + \left(1-\frac{\rho_1 \vep^2 }{\rho_0}\right)\left(\frac{(z(\vep))^2}{ c^2_0}K^{*,(2)} + \frac{(z(\vep))^3}{c^3_0} K^{*,(3)} + \mathcal R^{(1)}_{\textrm{res},\vep}\right) \right] l_\vep.
\end{align*}
From this, by utilizing \eqref{eq:145}, \eqref{def-l}, \eqref{eq:143}, \eqref{eq:144} and the fact that $ (1/2 \mathbb I+K_0)\mathcal Q$ is invertible in $\mathcal L(\mathcal Q(L^2(\Gamma)))$  and $\lim_{\vep\rightarrow 0}z(\vep) = 0$, we derive
\begin{align}\label{eq:146}
\|\mathcal Q l_\vep\|_{L^2(\Gamma)} \le C \max(\vep^2, |z(\vep)|^2).
\end{align}
 Applying the operator $\mathcal P$ to the both sides of equation \eqref{eq:143}, and using \eqref{eq:144}, \eqref{eq:146} and the identities $\left(1/2 \mathbb I +K^*_0\right) \mathcal P = 0$, $\mathcal P\left(1/2 \mathbb I +K^*_0\right) \mathcal Q = 0$ and $\mathbb I = \mathcal P + \mathcal Q$ gives 
\begin{align*}
\left(\frac{\vep^2 \rho_1}{\rho_0} + \frac{(z(\vep))^2}{c^2_0} \langle K^{*,(2)} S^{-1}_01, S^{-1}_01\rangle_{S_0} + \frac{(z(\vep))^3}{c^3_0}\langle K^{*,(3)} S^{-1}_01, S^{-1}_01\rangle_{S_0
} + \mathcal R^{(2)}_{\textrm{res},\vep}\right) S^{-1}_01  = 0,
\end{align*}
where $\mathcal R^{(2)}_{\textrm{res},\vep}$ satisfies
\begin{align*}
\left\|\mathcal R^{(2)}_{\textrm{res},\vep}\right\|_{L^2(\Gamma)} \le C \max(\vep^4, \vep^2|z(\vep)|^2).
\end{align*}
From this, we find that $z(\vep)$ and $\vep$ have the same order of magnitude relative to $\vep$ as $\vep$ approaches $0$. Thus, we can use Lemma \ref{le:6} to get
\begin{align}\label{eq:148}
\frac{\rho_1\vep^2}{\rho_0} - \frac{(z(\vep))^2}{\mathcal C_\Omega c_0^2}|\Omega|-\frac{i (z(\vep))^3 \mathcal|\Omega|}{4\pi c^3_0} + \mathcal R^{(3)}_{\textrm{res},\vep} = 0,
\end{align}
where $\mathcal R^{(3)}_{\textrm{res},\vep}$ satisfies
\begin{align} \label{eq:149}
\left|\mathcal R^{(3)}_{\textrm{res},\vep}\right| \le C\vep^4.
\end{align}
Recall that $\omega_M$ is defined in \eqref{eq:45}. Dividing by the constant $\omega_M^2|\Omega|\mathcal C^{-1}_\Omega c^{-2}_0$ on both sides of \eqref{eq:148}, we end up with the following characteristic equation for estimating the resonance:
\begin{align}\label{Final-characteristic-equation}
\vep^2 - \frac{(z(\vep))^2}{\omega^2_M}-i\frac{(z(\vep))^3 \mathcal C_\Omega}{4\pi c_0 \omega^2_M} + \frac{R^{(3)}_{\textrm{res},\vep}\mathcal C_\Omega c^2_0}{\omega^2_M|\Omega|} =0.
\end{align}
Note that $\mathcal R^{(3)}_{\textrm{res},\vep}$ satisfies \eqref{eq:149}. We look for the solution of the form $z(\vep)= \vep \beta_0 + \beta_1 \vep^2 + z_{res}(\vep) $ with $|z_{res}(\vep)| \le C\vep^3$. Plugging it into the equation \eqref{Final-characteristic-equation} and equating the terms of the same order of $\vep$, we get 
\begin{align*}
\beta_0:=\pm {\omega_M}, \quad \beta_1=-i\frac{\omega^2_M \mathcal C_\Omega }{8\pi c_0}.
\end{align*}
Therefore, with the aid of \eqref{eq:118}, for each sufficiently small $\vep>0$, we can find only two points $\widetilde z_{\pm}(\vep)$ where $Q_\vep(\widetilde z_{\pm}(\vep))$ fails to be injective and $\widetilde z_{\pm}(\vep)$ satisfy
\begin{align*}
\left|\widetilde z_{\pm}(\vep) \mp \omega_M \vep + i\frac{\omega^2_M \mathcal C_\Omega }{8\pi c_0} \vep^2 \right| \le C\vep^3, 
\end{align*}
whence the assertion of this statement follows from \eqref{eq:61} and the equivalence of Definition \ref{d2} and Definition \ref{d1}.
\end{proof}

\begin{remark}\label{re:1}
Statement \eqref{z2} of Lemma \ref{le:a1} implies that within any compact set $V\subset \CC$ containing $\pm \omega_M$, the resolvent $R^H_{\rho_\vep,k_\vep}(z)$ with sufficiently small $\vep>0$ exhibits two unique scattering resonances $z_{\pm}(\vep)$, both situated in the lower half complex plane. Furthermore, the two sequences of resonances $z_{\pm}(\vep)$ converge to $\pm \omega_M$, respectively, at the order of $\vep$ as the radius of the bubble $\vep$ tends to zero. 
\end{remark}

\section*{Acknowledgment} We thank Arpan Mukherjee and Soumen Senapati, from the Radon institute, for helpful discussions which largely inspired this work. This work is supported by the Austrian Science Fund (FWF) grant P: 36942.

\end{document}